\documentclass[11pt]{article}
\PassOptionsToPackage{obeyspaces}{url}

\usepackage[hidelinks]{hyperref}

\usepackage{bookmark}
\bookmarksetup{open,numbered,depth=5}
\usepackage{amsfonts}
\usepackage{amsmath}
\usepackage{amssymb, amsthm, enumitem}
\usepackage{pgf,tikz}
\usepackage{graphicx}
\usetikzlibrary{arrows.meta}
\usepackage{bm}
\usepackage{cancel}
\usepackage{comment}
\usetikzlibrary{calc}

\usepackage{etoolbox} 
\patchcmd{\thebibliography}{\leftmargin\labelwidth}{\leftmargin\labelwidth\addtolength\itemsep{-0.3\baselineskip}}{}{}

\usepackage{mathtools}
\usepackage{xstring}

\author{
Ting-Wei Chao\thanks{Department of Mathematics, Massachusetts Institute of Technology, Cambridge, MA, USA. \texttt{twchao@mit.edu}. }
\and Zichao Dong\thanks{Extremal Combinatorics and Probability Group (ECOPRO), Institute for Basic Science (IBS), Daejeon, South Korea. Supported by the Institute for Basic Science (IBS-R029-C4). \texttt{$\{$zichao,hongliu$\}$@ibs.re.kr}. }
\and Hong Liu\footnotemark[2]
\and Xichao Shu\thanks{Institute of Mathematics, Leipzig University, Leipzig, Germany. \texttt{xichao.shu@uni-leipzig.de}. }
\and Shuaichao Wang\thanks{Center for Combinatorics and LPMC, Nankai University, Tianjin 300071, China. Supported by the China Scholarship Council. \texttt{wsc17746316863@163.com}. }
}
\date{}

\title{A dichotomy for hypergraph Zarankiewicz problems on axis-parallel boxes}

\usepackage[nameinlink]{cleveref}

\newtheorem{theorem}{Theorem}[section]

\newtheorem{lemma}[theorem]{Lemma}

\newtheorem{claim}[theorem]{Claim}
\newtheorem{proposition}[theorem]{Proposition}

\newtheorem{observation}[theorem]{Observation}

\newtheorem{question}[theorem]{Question}

\theoremstyle{definition}

\newcommand*{\eqdef}{\stackrel{\mbox{\normalfont\tiny def}}{=}} 
\newcommand*{\veps}{\varepsilon}                                
\DeclarePairedDelimiter\abs{\lvert}{\rvert}                     

\newcommand*{\N}{\mathbb{N}}                                    
\newcommand*{\R}{\mathbb{R}}                                    

\newcommand*{\cH}{\mathcal{H}}

\newcommand*{\vF}{\vec{F}}
\newcommand*{\cB}{\mathcal{B}}
\newcommand*{\cA}{\mathcal{A}}
\newcommand*{\cS}{\mathcal{S}}
\newcommand*{\cT}{\mathcal{T}}
\newcommand*{\cX}{\mathcal{X}}
\newcommand*{\cY}{\mathcal{Y}}
\newcommand*{\vn}{\vec{n}}
\newcommand*{\wz}{\widetilde{z}}
\renewcommand*{\sc}{\mathsf{c}}

\definecolor{mred}{RGB}{217, 38, 38}
\definecolor{mgreen}{RGB}{38, 217, 38}
\definecolor{mblue}{RGB}{38, 38, 217}
\definecolor{mcyan}{RGB}{38, 128, 128}
\definecolor{mmagenta}{RGB}{128, 38, 128}
\definecolor{myellow}{RGB}{128, 128, 38}
\definecolor{mgray}{RGB}{98, 98, 98}

\allowdisplaybreaks[4]

\begin{document}

\maketitle

\begin{abstract}
 We study the Zarankiewicz problem for $r$-partite, $r$-uniform intersection hypergraphs arising from $r$ families of axis-parallel boxes in $\mathbb{R}^d$ with prescribed directions $F_1, \dots, F_r \subseteq \{1, \dots, d\}$. This extends the problems studied by Chan and Har-Peled on points and $d$-dimensional boxes in $\mathbb{R}^d$, corresponding to $(F_1,F_2)=(\varnothing,[d])$, as well as by Chan, Keller, and Smorodinsky on $r$ families of $d$-dimensional boxes, corresponding to $(F_1,\dots,F_r)=([d],\dots,[d])$.

Our main result establishes a sharp dichotomy for the Zarankiewicz number in this setting: it is either $\Theta_r(tn^{r-1})$ or at least $\Omega \bigl( tn^{r-1} \cdot \frac{\log n}{\log\log n} \bigr)$, depending only on a simple set-theoretic condition on $(F_1,\dots,F_r)$, which we call $2$-\emph{coherence}. Informally, $2$-coherence captures whether the configuration contains an underlying two-dimensional incidence structure, which is precisely what gives rise to the extra polylogarithmic factor. Our proof proceeds via a sequence of reductions and a geometric slicing argument that reduces the problem to planar incidence bounds.
\end{abstract}

\section{Introduction} \label{sec:intro}

In 1951, Zarankiewicz~\cite{zarankiewicz} asked for the maximum number $z(n; t)$ of edges in a bipartite graph on $2n$ vertices and parts of size $n$ containing no copy of $K_{t, t}$. The classical K\H{o}v\'{a}ri--S\'{o}s--Tur\'{a}n theorem~\cite{kovari_sos_turan} implies $z(n; t) = O \bigl( n^{2 - \frac{1}{t}} \bigr)$. Determining its tightness remains a major open problem, though several algebraic incidence constructions~\cite{kollar_ronyai_szabo,alon_ronyai_szabo,bukh2015,bukh2024} show it is tight for $K_{s, t}$ in place of $K_{t, t}$ when $s \gg t$. 

In 1964, Erd\H{o}s~\cite{erdos1964} proposed a natural hypergraph analogue: determine the maximum number $z_r(n; t)$ of hyperedges in an $r$-uniform, $r$-partite hypergraph on $rn$ vertices with parts of size $n$ that contains no copy of $K^{(r)}_{t, \dots, t}$. In the same paper, he proved the upper bound $z_r(n; t) = O \bigl( n^{r - \frac{1}{t^{r-1}}} \bigr)$, which generalizes the K\H{o}v\'{a}ri--S\'{o}s--Tur\'{a}n theorem. Analogous tightness results~\cite{ma_yuan_zhang, chen_liu_ye} are known when the largest parameter $t_r$ is much larger than $t_1, \dots, t_{r-1}$, for $K_{t_1, \dots, t_r}^{(r)}$ in place of $K_{t, \dots, t}^{(r)}$. 

\paragraph{Zarankiewicz problems on axis-parallel boxes.} Bipartite graphs naturally model geometric incidences or intersections. Consider the incidence graph $G$ between $n$ points and $n$ lines, where an edge connects a point $p$ and a line $\ell$ if $p \in \ell$. Since two points determine a unique line, $G$ contains no $K_{2,2}$. The K\H{o}v\'{a}ri--S\'{o}s--Tur\'{a}n theorem then yields $e(G) = O(n^{3/2})$, while the Szemer\'{e}di--Trotter theorem~\cite{szemeredi_trotter}, a central result in incidence geometry, improves this bound to the tight $O(n^{4/3})$. The incidence problems are closely related to algorithmic problems in theoretical computer science such as range searching \cite{agarwal_erickson} and the so-called Hopcroft's problem \cite{chan_zheng}. We refer to \cite{mustafa_pach,chan_har-peled,chan_keller_smorodinsky} for related studies on simplices, balls, pseudo-disks, and other geometric objects. 

This paper studies Zarankiewicz problems for incidences between axis-parallel boxes. For any positive integer $n$, we use the standard notation $[n] \eqdef \{1, \dots, n\}$. Throughout the whole paper, every \emph{box} refers to an \emph{axis-parallel box}. We require several notations. 

In the $d$-dimensional Euclidean space $\R^d$, an \emph{axis-parallel box} is the product of $d$ closed intervals $\bm{b} = [a_1, b_1] \times \dots \times [a_d, b_d]$, where $a_i \le b_i \, (i = 1, \dots, d)$ are real numbers. If $a_i = b_i$, then $\bm{b}$ \emph{collapses} along the $i$-th coordinate axis. We call the set $\bigl\{ i \in [d] : a_i < b_i \bigr\}$ the \emph{direction set} of $\bm{b}$. For example, the box $\{0\} \times [e, \pi] \times \{\frac{2}{3}\}$ in $\R^3$ collapses in the $1$-st and $3$-rd coordinates and has direction set $\{2\}$. In the literature, axis-parallel boxes usually exclude the case $a_i = b_i$. However, we allow such degeneracies and study a strictly more general class of configurations. We call boxes with direction set $[d]$ (i.e., without any $a_i = b_i$ for $i = 1, \dots, d$) \emph{full boxes}. 

Let $d \ge 1, \, r \ge 2$. In $\R^d$, every $r$-\emph{direction}-\emph{vector} $\vF \eqdef (F_1, \dots, F_r)$ refers to an $r$-tuple of subsets of $[d]$. An $r$-tuple of sets\footnote{Hereafter boxes in each set are not necessarily distinct, and so formally all the box sets we consider are multisets.} of boxes $(\cB_1, \dots, \cB_r)$ is an $\vF$-\emph{family} if each box in $\cB_j$ has direction set $F_j$. For $\vF$-family $(\cB_1, \dots, \cB_r)$, we study the $r$-uniform $r$-partite intersection hypergraph $\cH = (V, E)$, where $V \eqdef \cB_1 \cup \dots \cup \cB_r$ and $\{\bm{b}_1, \dots, \bm{b}_r\} \in E$ if and only if $\bm{b}_1 \cap \dots \cap \bm{b}_r \ne \varnothing$. For $r$-\emph{size}-\emph{vector} $\vn = (n_1, \dots, n_r)$ of positive integers, denote by $z_d^{\vF}(\vn; t) = z_d^{\vF}(n_1, \dots, n_r; t)$ the maximum number of edges in $\cH$, provided that $|\cB_j| = n_j$ for $j = 1, \dots, r$ and $\cH$ forbids the complete $r$-partite subgraph $K_{t, \dots, t}^{(r)}$. We also introduce the abbreviated notation $z_d^{\vF}(n; t) \eqdef z_d^{\vF} \bigl( \underbrace{n, \dots, n}_r; t \bigr)$. 

The Zarankiewicz problem between points and full boxes in $\R^d$, where $\vF = (\varnothing, [d])$, has been extensively studied. In the complexity analysis of orthogonal range reporting, Chazelle~\cite{chazelle} constructed random points and boxes in $\R^d$ to deduce $z_d^{\vF}(n; 2) = \Omega \bigl( n(\frac{\log n}{\log\log n})^{d-1} \bigr)$. This construction was later reformulated in mathematical language and generalized by Tomon~\cite{tomon} and by Chan--Har-Peled~\cite{chan_har-peled} to $z_d^{\vF}(n; t) = \Omega \bigl( tn(\frac{\log n}{\log\log n})^{d-1} \bigr)$. Basit, Chernikov, Starchenko, Tao, and Tran~\cite{basit_etal} obtained the upper bound $z_d^{\vF}(n; t) = O_{d, t} \bigl( n (\log n)^{2d} \bigr)$. This upper bound was later improved by Chan--Har-Peled~\cite{chan_har-peled} to $z_d^{\vF}(n; t) = O_d \bigl( tn(\frac{\log n}{\log\log n})^{d-1} \bigr)$, matching the lower bound and thereby resolving the problem. 

The Zarankiewicz problem for full boxes has also attracted much attention. Keller and Smorodinsky~\cite{keller_smorodinsky} established $z_2^{\vF}(n;t) = \Theta_t \bigl( n \frac{\log n}{\log\log n}\bigr)$ for rectangles, where $\vF = ([2], [2])$. Chan, Keller, and Smorodinsky~\cite{chan_keller_smorodinsky_arxiv} later extended this to incidence $r$-hypergraphs in $\R^d$ and improved the dependence on $t$ to linear, obtaining the sharp bound $z_d^{\vF}(n; t) = \Theta_{r, d} \bigl( tn^{r-1}(\frac{\log n}{\log\log n})^{d-1} \bigr)$ for $\vF = \bigl( \underbrace{[d], \dots, [d]}_{r} \bigr)$. 

For an arbitrary $\vF$, the lower bound $z_d^{\vF}(n; t) \ge (t-1) \cdot n^{r-1}$ is straightforward. Indeed, suppose each box in $\cB_1, \dots, \cB_{r-1}$ contains the origin. In $\cB_r$, choose $t-1$ boxes containing the origin and place the remaining $n - t + 1$ boxes far away so that they intersect no others. Then the intersection hypergraph is $K_{n, \dots, n, t-1}^{(r)}$, which contains no $K_{t, \dots, t}^{(r)}$, and thus yields $(t-1)\, n^{r-1}$ hyperedges. 

In some cases, this lower bound is tight. For example, for intervals $(\{1\}, \{1\})$ in $\R^1$ and for planar vertical and horizontal segments $(\{1\}, \{2\})$ in $\R^2$, it was previously established in~\cite{chan_har-peled,keller_smorodinsky} that
$z_1^{(\{1\}, \{1\})}(n; t) = \Theta(tn)$ and $z_2^{(\{1\}, \{2\})}(n; t) = \Theta(tn)$. On the other hand, an additional logarithmic factor is sometimes unavoidable: for the aforementioned $(\varnothing, [2])$ in $\R^2$ and $([3], [3], [3])$ in $\R^3$, it is known~\cite{chan_har-peled,chan_keller_smorodinsky} that
$z_2^{(\varnothing, [2])}(n; t) = \Theta \bigl( tn \tfrac{\log n}{\log\log n} \bigr)$ and
$z_3^{([3], [3], [3])}(n; t) = \Theta \bigl( tn^2 (\tfrac{\log n}{\log\log n})^2 \bigr)$. 

\smallskip

Motivated by these results on $z_d^{\vF}(n; t)$ for different direction-vectors $\vF$, a natural question arises:
\begin{center}
    \emph{Under what conditions on $\vF$ does a polylogarithmic factor emerge in $z_d^{\vF}(n; t)$?}
\end{center}
Our main finding is that this phenomenon admits a sharp dichotomy depending only on a simple set-theoretic property of the direction sets. Informally, a polylogarithmic factor emerges in $z_d^{\vF}(n; t)$ when all but at most one of the families share at least two common coordinate directions. Formally, we refer to an $r$-direction-vector $\vF = (F_1, \dots, F_r)$ in $\R^d$ as $2$-\emph{coherent} if there is some $k \in [r]$ such that $\bigl| \bigcap_{j \in [r] \setminus \{k\}} F_j \bigr| \ge 2$, and \emph{non}-$2$-\emph{coherent} otherwise. The theorem below formalizes this dichotomy. 

\begin{theorem} \label{thm:main}
    Let $d \ge 1, \, r \ge 2, \, t \ge 2$ be integers. For any $r$-direction-vector $\vF$ in $\R^d$, we have
    \[
    z_d^{\vF} (n; t) = 
    \begin{cases}
        \Theta_r (t n^{r-1}) &\text{if $\vF$ is non-$2$-coherent}, \\
        \Omega \bigl( t n^{r-1} \frac{\log n}{\log\log n} \bigr) \qquad &\text{if $\vF$ is $2$-coherent}. 
    \end{cases}
\]
\end{theorem}

In particular, this yields a complete characterization of those direction-vectors for which the trivial lower bound $\Omega(tn^{r-1})$ is asymptotically optimal, identifying $2$-coherence as the threshold for the emergence of superlinear behavior. The four examples above fit this dichotomy perfectly: the interval case $(\{1\}, \{1\})$ in $\R^1$ and the planar vertical-horizontal segment case $(\{1\}, \{2\})$ in $\R^2$ are non-$2$-coherent and satisfy the linear bound $\Theta(tn)$, whereas the point-box case $(\varnothing, [2])$ in $\R^2$ and the full-box case $([3], [3], [3])$ in $\R^3$ are $2$-coherent and exhibit an extra polylogarithmic factor.

At a high level, our proof reveals that the appearance of a polylogarithmic factor is driven by an underlying two-dimensional incidence structure. When $\vF$ is $2$-coherent, such a structure is explicitly present, allowing us to embed planar constructions with polylogarithmic growth and obtain the lower bound. In contrast, the main difficulty lies in the non-$2$-coherent case, where no such structure is a priori visible. To overcome this, we develop a sequence of reductions that gradually impose geometric structure on the problem. A key role is played by the \emph{canonical} case $\vF = ([d] \setminus \{1\}, \dots, [d] \setminus \{d\})$, where each box in $\cB_i$ is a $(d-1)$-dimensional box perpendicular to the $i$-th coordinate direction. This case extends earlier work of Chan--Keller--Smorodinsky~\cite{keller_smorodinsky,chan_keller_smorodinsky} on vertical and horizontal segments in $\R^2$, corresponding to $z_2^{(\{1\}, \{2\})}(n; t)$, and serves as a central intermediate step in our argument. The key technical step is a geometric slicing argument that decomposes high-dimensional intersections into collections of planar incidence problems, to which sharp two-dimensional bounds can be applied. This ultimately allows us to recover the optimal $O(tn^{r-1})$ upper bound. We defer a more detailed proof sketch to~\Cref{sec:upper_sketch}. This reduction framework, which isolates and exploits hidden low-dimensional structures, may be of independent interest for other geometric Zarankiewicz-type problems.

\paragraph{Related Zarankiewicz problems.}
When $\vF$ is non-$2$-coherent, our result shows that the upper bound matches the trivial lower bound asymptotically, hence yields a tight characterization of the extremal behavior in this regime. This contrasts with the general theory of semi-algebraic incidence graphs, where one typically cannot obtain bounds of order $n^{r-1}$. Recall that a bipartite graph $G$ is \emph{semi-algebraic} in $\R^d$ if its vertex sets are point sets $P, Q \subseteq \R^d$, and an edge $(p, q)$ is present if and only if it satisfies a Boolean combination of finitely many polynomial equalities and inequalities. For such graphs, in a seminal work Fox, Pach, Sheffer, Suk, and Zahl~\cite{fox_pach_sheffer_suk_zahl} proved that if $G$ is $K_{t, t}$-free, then $e(G) = O \Bigl( n^{\frac{2d}{d+1} + \veps} \Bigr)$, independently of $t$. The $\varepsilon$-loss was later removed by Tidor and Yu~\cite{tidor_yu}, who showed $e(G) = O \Bigl( n^{\frac{2d}{d+1}} \Bigr)$, with extensions to incidence hypergraphs as well. These bounds remain significantly larger than the linear-type bound $O(tn^{r-1})$ in our non-$2$-coherent setting, reflecting the greater rigidity of incidences involving axis-parallel boxes. 

\paragraph{Paper organization.} We first present the lower bound constructions with logarithmic factors in \Cref{sec:lower_proof}. The deduction of the upper bound in \Cref{thm:main} proceeds via several reductions: we introduce the necessary terminology and outline the main ideas in \Cref{sec:upper_sketch}, followed by full proofs in \Cref{sec:upper_proof}. We conclude with future directions for Zarankiewicz box-problems in \Cref{sec:remark}. 

\section{The lower bound constructions} \label{sec:lower_proof}

In this section, we prove the lower bound in Theorem \ref{thm:main} for the $2$-coherent case. 

\begin{theorem} \label{thm:lower}
    Let $d \ge 2, \, r \ge 2, \, t \ge 2$ be integers and $\vF$ be a $2$-coherent $r$-direction-vector in $\R^d$. Then $z_d^{\vF}(n; t) = \Omega \bigl( tn^{r-1} \frac{\log n}{\log\log n} \bigr)$ holds for every positive integer $n$. 
\end{theorem}

We remark that \Cref{thm:lower} assumes $d \ge 2$ since no $2$-coherent $r$-direction-vector exists when $d = 1$.
Our construction is based on the following known construction in $\R^2$ for the $2$-direction-vector $\vF \eqdef (\varnothing, \{1, 2\})$.
Constructions in higher dimensions were obtained by Chazelle~\cite{chazelle} in his study of data structures. Indeed, he constructed point sets and $d$-dimensional boxes in $\R^d$ to show that $z_d^{\vF}(n; t) = \Omega_d \Bigl( tn \bigl( \frac{\log n}{\log\log n} \bigr)^{d-1} \Bigr)$. Such constructions were later reformulated more explicitly by Chan and Har-Peled~\cite[Appendix~B]{chan_har-peled} and Tomon~\cite{tomon}. For our purposes, it suffices to consider the simpler two-dimensional constructions, and we therefore refer to the work of Basit et al.~\cite{basit_etal}.

\begin{theorem} \label{thm:point_rect}
   Suppose $t \ge 2$ and $n \in \N_+$. In $\R^2$, there exist a set $\cB_1$ of $n$ points and a set $\cB_2$ of $n$ full boxes, such that the intersection graph $G$ between $\cB_1, \cB_2$ is $K_{t, t}$-free with $e(G) = \Omega \bigl( tn\frac{\log n}{\log\log n} \bigr)$.
\end{theorem}

In~\cite{basit_etal}, Basit et al.~only studied $K_{2, 2}$-free intersection graphs $G$ between $\cB_1$ and $\cB_2$, deducing $e(G) = \Omega \bigl( n \frac{\log n}{\log\log n} \bigr)$. By taking $(t - 1)$ identical copies of their families $\cB_1$ and $\cB_2$, each of size $\frac{n}{t-1}$ instead of $n$, we obtain the desired $K_{t, t}$-free construction stated in \Cref{thm:point_rect}.

We remark that the construction in~\cite{basit_etal} is given for open boxes; however, by shrinking each box slightly, we may derive the construction in our setting where the boxes are closed. 

To establish \Cref{thm:lower}, we first extend the construction in \Cref{thm:point_rect} to arbitrary $2$-coherent $r$-direction-vectors $\vF$ in $\mathbb{R}^2$, and then lift it to higher dimensions.

\begin{proof}[Proof of \Cref{thm:lower} assuming $d = 2$]
    Since $\vF = (F_1, \dots, F_r)$ is $2$-coherent in $\R^2$, it follows from the definition that all but at most one of $F_1, \dots, F_r$ must be $\{1, 2\}$. Assume without loss of generality that $F_2 = \dots = F_r = \{1, 2\}$. We construct
    an $\vF$-family $(\cB_1, \dots, \cB_r)$ with $|\cB_j| = n$ for all $j$ such that the intersection hypergraph $\cH$ is $K^{(r)}_{t, \dots, t}$-free with $e(\cH) = \Omega \bigl( tn^{r-1} \frac{\log n}{\log\log n} \bigr)$. 

    Let $(\cB_1', \cB_2')$ be the $(\varnothing, \{1, 2\})$-family of $n$ points and $n$ boxes given by \Cref{thm:point_rect} and denote $\cH'$ as the intersection hypergraph of $(\cB_1', \cB_2')$. Set $\cB_2 \eqdef \cB_2'$. 
    \vspace{-0.5em}
    \begin{itemize}
        \item If $|F_1| = 0$, then $F_1 = \varnothing$, and we take $\cB_1 \eqdef \cB_1'$. 
        \vspace{-0.5em}
        \item If $|F_1| = 1$, then $F_1 = \{1\}$ (or $\{2\}$), and we construct $\mathcal{B}_1$ by replacing each point of $\mathcal{B}_1'$ with a tiny segment in the first (resp.~second) coordinate direction containing that point. 
        \vspace{-0.5em}
        \item If $|F_1| = 2$, then $F_1 = \{1, 2\}$, and we construct $\cB_1$ by replacing each point in $\cB_1'$ with a tiny axis-parallel square containing that point. 
    \end{itemize}
    \vspace{-0.5em}
    Here, ``tiny'' means that the geometric objects are chosen sufficiently small so that the intersection graph between $\cB_1$ and $\cB_2$ coincides with $\cH'$.
    Let $\cB_3, \cB_4, \dots, \cB_r$ be $r - 2$ collections, each consisting of $n$ identical large full boxes, such that any box in these collections contains every box in $\mathcal{B}_1 \cup \mathcal{B}_2$. Consider the intersection hypergraph $\cH$ of $(\cB_1, \dots, \cB_r)$. The hyperedges of $\cH$ are exactly
    \[
    \bigl\{ \{\bm{b}_1, \dots, \bm{b}_r\} : \{\bm{b}_1, \bm{b}_2\} \in e(\cH'), \, \bm{b}_i \in \cB_i \, (\forall i \in [r]) \bigr\}.
    \]
    Since $\cH'$ is $K_{t, t}$-free, we know that $\cH$ is $K^{(r)}_{t, \dots, t}$-free. Also, 
    \[
    e(\cH) = e(\cH') \cdot n^{r-2} = \Omega \bigl( tn^{r-1} \tfrac{\log n}{\log\log n} \bigr). 
    \]
    The proof of \Cref{thm:lower} in the case $d = 2$ is complete. 
\end{proof}

Now, we are ready to prove \Cref{thm:lower} in its full generality. 

\begin{proof}[Proof of \Cref{thm:lower}]
    Since $\vF = (F_1, \dots, F_r)$ is $2$-coherent in $\R^d$, we may assume without loss of generality that $\{1, 2\} \subseteq F_2 \cap F_3 \cap \dots \cap F_r$. For $j = 1, \dots, r$, we write $F_j' \eqdef F_j \cap \{1, 2\}$ and define $\vF' \eqdef (F_1', \dots, F_r')$. Then $\vF'$ is a $2$-coherent $r$-direction-vector in $\R^2$. 
    
    According to what we have proved in the $d = 2$ case, there exists an $\vF'$-family $(\cB'_1, \dots, \cB'_r)$ in $\R^2$ with $|\cB_j'| = n$ whose intersection hypergraph $\cH'$ is $K_{t, \dots, t}^{(r)}$-free. To obtain $(\cB_1, \dots, \cB_r)$, we first embed $(\cB'_1, \dots, \cB'_r)$ into the plane $\{x_3 = \dots = x_d = 0\} \subseteq \R^d$, and then extend each box to make it an $\vF$-family. For each $j \in [r]$ and every $\bm{b}_j \in \cB'_j$, we extend $\bm{b}_j$ along the directions $F_j \setminus \{1, 2\}$ by segments of arbitrary lengths. Then the intersection hypergraph $\cH$ behind $(\cB_1, \dots, \cB_r)$ coincides with $\cH'$. Therefore, $\cH$ is $K_{t, \dots, t}^{(r)}$-free as well with $e(\cH) = e(\cH') = \Omega \bigl( tn^{r-1} \frac{\log n}{\log\log n} \bigr)$. 
\end{proof}

\section{The upper bound---overview of ideas} \label{sec:upper_sketch}

For every $r$-size-vector $\vn = (n_1, \dots, n_r)$ of positive integers and every integer $t \ge 2$, write 
\[
g_t(\vn) = g_t(n_1, \dots, n_r) \eqdef t n_1 \cdots n_r \Bigl( \frac{1}{n_1} + \dots + \frac{1}{n_r} \Bigr). 
\]
Then $g_t\bigl( \underbrace{n, \dots, n}_r \bigr) = tr n^{r-1}$. 
We will prove a strengthening of the upper bounds in \Cref{thm:main}.

\begin{theorem} \label{thm:upper_general}
    Let $d \ge 1, \, r \ge 2, \, t \ge 2$, and $\vF$ be a non-$2$-coherent $r$-direction-vector in $\R^d$. Then there exists a constant $\Lambda_r > 0$ such that $z_d^{\vF}(\vn; t) \le \Lambda_r \cdot g_t(\vn)$ holds for every $r$-size-vector $\vn$. 
\end{theorem}

As noted in \Cref{sec:intro}, the theorem is tight up to a constant factor in the $n_1 = \dots = n_r$ balanced case. In general, \Cref{thm:upper_general} is again tight up to a constant factor, as shown by the following construction. 
Partition $\cB_1 \cup \cdots \cup \cB_r$ into $r$ disjoint groups so that, for each $j \in [r]$, the $j$-th group contains $\frac{n_j}{r}$ boxes from $\cB_j$. In the $j$-th group, arrange $n_j - t$ boxes from $\cB_j$ to be far away so that they are disjoint with others, while every other box from $\cB_1 \cup \dots \cup \cB_r$ in the same group contains a fixed point. Summing over all groups yields $\sum_{j=1}^ r\frac{n_1}{r} \cdots \frac{n_{j-1}}{r} \cdot t \cdot \frac{n_{j+1}}{r} \cdots \frac{n_r}{r} = \frac{g_t(\vn)}{r^{r-1}}$ hyperedges.

Our proof proceeds through several reduction steps, ultimately reducing the problem to a planar incidence bound.
The following theorem will be an important ingredient of our proof.

\begin{theorem}[{\cite[Theorem 1.5]{fox_pach}}] \label{thm:CKS}
    In the Euclidean plane $\R^2$, let $A$ be a multiset of $m$ horizontal segments and $B$ be a multiset of $n$ vertical segments. If the bipartite intersection graph $G$ between $A$ and $B$ is $K_{t, t}$-free with $t \ge 2$, then the maximum number of edges in $G$ is $\Theta(tm + tn)$. 
\end{theorem}

As far as we know, \Cref{thm:CKS} was first implicitly discovered by Fox and Pach~\cite{fox_pach}, where it was formulated for string graphs under the assumption that any two strings intersect in only a bounded number of points, and presented as a Turán-type result. The Zarankiewicz-type problem for axis-parallel segments was studied explicitly in~\cite{keller_smorodinsky}, where Keller--Smorodinsky established an $O(t^6 m + tn)$ upper bound and conjectured the sharp bound $\Theta(tm + tn)$. This upper bound was subsequently improved by Chan--Keller--Smorodinsky in~\cite{chan_keller_smorodinsky} to $O(t^2 m + tn)$. A complete resolution, stated explicitly as \Cref{thm:CKS}, was later obtained in~\cite{chan_keller_smorodinsky_arxiv}, an updated version of~\cite{chan_keller_smorodinsky}. 

We remark that the Fox--Pach~\cite{fox_pach} separator argument is very different from the later charging argument developed in~\cite{chan_keller_smorodinsky_arxiv} (inspired by Chalermsook--Orgo--Zarsav~\cite{chalermsook_orgo_zarsav}).
Very recently, \Cref{thm:CKS} was reproved via a newly-developed technique by Hunter--Milojević--Sudakov--Tomon~\cite{hunter_milojevic_sudakov_tomon}.

To prove \Cref{thm:upper_general}, we proceed through four reduction steps, ultimately reducing the problem to \Cref{thm:CKS} (see \hyperlink{figone}{Figure~1}). In this section, we introduce the necessary terminology and outline the main ideas, deferring the formal statements and proofs to \Cref{sec:upper_proof}.

\centerline{
\begin{tikzpicture}
    \clip (-0.5, -3) rectangle (16.4, 2.3);
    \draw (0, -0.8) rectangle (2.7, 0.8);
    \node at (1.35, 0.333) {\scriptsize the general case};
    \node at (1.35, -0.333) {\scriptsize (\Cref{thm:upper_general})};
    \draw (3.3, -0.8) rectangle (6, 0.8);
    \node at (4.65, 0.333) {\scriptsize the canonical case};
    \node at (4.65, -0.333) {\scriptsize (\Cref{prop:upper_canonical})};
    \draw (6.6, -0.8) rectangle (9.3, 0.8);
    \node at (7.95, 0.333) {\scriptsize the separated case};
    \node at (7.95, -0.333) {\scriptsize (\Cref{prop:upper_separated})};
    \draw (9.9, -0.8) rectangle (12.6, 0.8);
    \node at (11.25, 0.333) {\scriptsize the restricted case};
    \node at (11.25, -0.333) {\scriptsize (\Cref{prop:upper_restricted})};
    \draw (13.2, -0.8) rectangle (15.9, 0.8);
    \node at (14.55, 0.333) {\scriptsize known result};
    \node at (14.55, -0.333) {\scriptsize (\Cref{thm:CKS})};
    \draw[-Stealth][dash pattern = on 3pt off 3pt] (1.8, 0.8) arc (135:45:1.697);
    \node at (3, 1.6) {\scriptsize Reduction {\MakeUppercase{\romannumeral 1}}};
    \draw[-Stealth][dash pattern = on 3pt off 3pt] (5.1, 0.8) arc (135:45:1.697);
    \node at (6.3, 1.6) {\scriptsize Reduction {\MakeUppercase{\romannumeral 2}}};
    \draw[-Stealth][dash pattern = on 3pt off 3pt] (8.4, 0.8) arc (135:45:1.697);
    \node at (9.6, 1.6) {\scriptsize Reduction {\MakeUppercase{\romannumeral 3}}};
    \draw[-Stealth][dash pattern = on 3pt off 3pt] (11.7, 0.8) arc (135:45:1.697);
    \node at (12.9, 1.6) {\scriptsize Reduction {\MakeUppercase{\romannumeral 4}}};
    \draw[-Stealth] (14.1, -0.8) arc (-45:-135:1.697);
    \node at (12.9, -1.6) {\scriptsize \Cref{sec:upper_restricted}};
    \draw[-Stealth] (10.8, -0.8) arc (-45:-135:1.697);
    \node at (9.6, -1.6) {\scriptsize \Cref{sec:upper_separated}};
    \draw[-Stealth] (7.5, -0.8) arc (-45:-135:1.697);
    \node at (6.3, -1.6) {\scriptsize \Cref{sec:upper_canonical}};
    \draw[-Stealth] (4.2, -0.8) arc (-45:-135:1.697);
    \node at (3, -1.6) {\scriptsize \Cref{sec:upper_general}};
    \node at (7.95, -2.5) {\textbf{\hypertarget{figone}{Figure 1:}} The proof outline of \Cref{thm:upper_general}. };
\end{tikzpicture}
}

\paragraph{Reduction \texorpdfstring{{\MakeUppercase{\romannumeral 1}}}{I}.} In $\R^d$, define the \emph{canonical} direction-vector as
\[
\vF_d \eqdef \bigl( [d] \setminus \{1\}, [d] \setminus \{2\}, \dots, [d] \setminus \{d\} \bigr). 
\]
For instance, $\vF_3 = \bigl\{ \{2, 3\}, \{3, 1\}, \{1, 2\} \bigr\}$, and an $\vF_3$-family consists of collections $\cB_1, \cB_2, \cB_3$, where
\vspace{-0.5em}
\begin{itemize}
    \item $\cB_1$ consists of rectangles perpendicular to the $x$-axis, 
    \vspace{-0.5em}
    \item $\cB_2$ consists of rectangles perpendicular to the $y$-axis, 
    \vspace{-0.5em}
    \item $\cB_3$ consists of rectangles perpendicular to the $z$-axis. 
\end{itemize}
\vspace{-0.5em}

For simplicity, we set $z_d(\vn; t) \eqdef z_d^{\vF_d}(\vn; t)$, and we call an $\vF_d$-family a \emph{canonical} family. 
To obtain upper bounds on $z_d^{\vF}(\vn; t)$ in general, the canonical case $\vF = \vF_d$ turns out to be crucial. We will reduce the general case to bounding $z_d(\vn; t)$ for canonical families in \Cref{sec:upper_general}.

\paragraph{Reduction \texorpdfstring{{\MakeUppercase{\romannumeral 2}}}{II}.} 
For a canonical family $(\cB_1, \dots, \cB_d)$, each box $\bm{b}\in \cB_i$ lies in a hyperplane perpendicular to the $i$-th coordinate. Denote this hyperplane by $\Gamma_{\bm{b}}$. We say that a family $(\cB_1, \dots, \cB_d)$ is \emph{separated} if it is canonical, and the hyperplanes $\Gamma_{\bm{b}}$ with $\bm{b} \in \cB_i$ are all distinct for each $i \in [d]$. 

For every $\vn = (n_1, \dots, n_d)$, we denote by $z_d^*(\vn; t)$ the maximum of $e(\cH)$, given that $\cH$ is the intersection hypergraph of a separated family $(\cB_1, \dots, \cB_d)$ with $|\cB_j| = n_j$ for $j = 1, \dots, d$, and $\cH$ forbids the subgraph $K_{t, \dots, t}^{(d)}$. 
We will establish $z_d(\vn; t) = z_d^*(\vn; t)$ in \Cref{sec:upper_canonical}.

\paragraph{Reduction \texorpdfstring{{\MakeUppercase{\romannumeral 3}}}{III}.} We say that a family $(\cB_1, \dots, \cB_d)$ is \emph{restricted} if it is separated, and
\[
\text{$\bigcap_{i=1}^{d-1} \bm{b}_i \ne \varnothing$ holds for every $(\bm{b}_1, \dots, \bm{b}_{d-1}) \in \cB_1 \times \dots \times \cB_{d-1}$}. 
\]

For every $\vn = (n_1, \dots, n_d)$, we denote by $\wz_d^*(\vn; t)$ the maximum of $e(\cH)$, given that $\cH$ is the intersection hypergraph of a restricted family $(\cB_1, \dots, \cB_d)$, where $|\cB_j| = n_j \, (j = 1, \dots, d)$ and $\cH$ forbids $K_{t, \dots, t}^{(d)}$. 
In \Cref{sec:upper_separated}, we will establish $z^*_d(\vn; t) = O_d \bigl( g_t(\vn) \bigr)$, given that $\wz_d^*(\vn; t) = O_d \bigl( g_t(\vn) \bigr)$. 

\paragraph{Reduction \texorpdfstring{{\MakeUppercase{\romannumeral 4}}}{IV}.} 
It follows from the definitions that any separated $\vF_2$-family is restricted. In fact, when $d = 2$, the restriction on the first $d - 1 = 1$ families holds vacuously. So, \Cref{thm:CKS} implies that $\wz_2^*(m, n; t) = z_2^*(m, n; t) = O(tm + tn)$, which then implies that $\wz_2^*(\vn; t) = O \bigl( g_t(\vn) \bigr)$. 

Using a novel geometric argument, we are going to establish $\wz_d^*(\vn; t) = O_d \bigl( g_t(\vn) \bigr)$ in \Cref{sec:upper_restricted} by reducing it to the case $d = 2$.

\section{The upper bound---proof in detail} \label{sec:upper_proof}

Following the strategy and reductive steps introduced in \Cref{sec:upper_sketch}, we establish the upper bound in \Cref{thm:main} in this section. To be specific, our goal is to deduce \Cref{thm:upper_general} in four steps. 

\subsection{The restricted case} \label{sec:upper_restricted}

Firstly, we prove \Cref{thm:upper_general} in the restricted case. 

\begin{proposition} \label{prop:upper_restricted}
    Suppose $d \ge 2$ and $t \ge 2$. Then $\wz_d^*(\vn; t) \le 27 g_t(\vn)$ holds for every $d$-size-vector $\vn$.
\end{proposition}

To prove \Cref{prop:upper_restricted}, we use the following folklore fact that axis-parallel boxes have Helly number $2$.
This fact will be useful in \Cref{append:upper_onedim} as well. 

\begin{observation} \label{obs:box_helly}
    Suppose $d \ge 1, \, r \ge 1$ and let $\bm{b}_1, \dots, \bm{b}_r$ be boxes in $\R^d$ such that $\bm{b}_i \cap \bm{b}_j \ne \varnothing$ for every pair of distinct indices $i, j$.
    Then $\bm{b}_1 \cap \dots \cap \bm{b}_r \ne \varnothing$. 
\end{observation}

\begin{proof}
    When $d = 1$, this is exactly Helly's theorem in one dimension. For a larger $d$, the projection intervals of $\bm{b}_1, \dots, \bm{b}_r$ to every coordinate axis have pairwise intersection, and so there exists a point piercing all such intervals. So, the Cartesian product of these points pierces each of $\bm{b}_1, \dots, \bm{b}_r$. 
\end{proof}

We will use a more precise version of \Cref{thm:CKS} due to Chan--Keller--Smorodinsky~\cite{chan_keller_smorodinsky_arxiv}. 

\begin{lemma}[{\cite[Proposition 2.4]{chan_keller_smorodinsky_arxiv}}] \label{lem:COZ_CKS}
    In $\R^2$, let $\cA$ and $\cB$ be multisets of $m$ horizontal segments and $n$ vertical segments, respectively. If their intersection graph $G$ is $K_{t, t}$-free, then $e(G) \le 27t(m + n)$. 
\end{lemma}

Our proof of \Cref{prop:upper_restricted} proceeds via a geometric reduction from general dimension $d$ to the planar case $d = 2$ (see \Cref{lem:COZ_CKS}). Since the main ideas are most transparent when $d = 3$, we first present the argument in this special case and then extend it to arbitrary $d$.

\begin{proof}[Proof of \Cref{prop:upper_restricted} assuming $d = 3$]
    Suppose $t \ge 2$ and let $(\cB_1, \cB_2, \cB_3)$ be a restricted family with $|\cB_i| = n_i$ for $i = 1, 2, 3$. Assume that $\cH$ is the intersection hypergraph of $(\cB_1, \cB_2, \cB_3)$. Recall that $\wz_3^*(n_1, n_2, n_3; t)$ denotes the maximum possible value of $e(\cH)$, provided that every pair of boxes $\bm{b}_1 \in \cB_1, \, \bm{b}_2 \in \cB_2$ intersect and that $\cH$ contains no $K_{t, t, t}^{(3)}$. 
    
    In the following, we will use $x,y,z$ to denote the first, second, and third coordinate of $\mathbb{R}^3$ respectively.
    Recall that a restricted family is always separated. Therefore, 
    as a crucial reduction,
    we may assume without loss of generality that the following holds. 
    \vspace{-0.5em}
    \begin{itemize}
        \item The plane $\{x = u\}$ contains exactly one box from $\cB_1$ for all $u\in [n_1]$, and the plane $\{y = v\}$ contains exactly one box from $\cB_2$ for all $v\in [n_2]$.
    \end{itemize}
    \vspace{-0.5em}

    Formally, applying a piecewise linear transformation along any coordinate axis to all boxes in $\mathcal{B}_1 \cup \mathcal{B}_2 \cup \mathcal{B}_3$ preserves their intersection hypergraph.
    Thus, we may rescale the $x$-coordinates and the $y$-coordinates by piecewise linear functions to ensure that the 
    boxes in $\mathcal{B}_1$ and $\mathcal{B}_2$ are evenly spaced. We can do this in such a way
    that, after projecting onto the $xy$-plane, the complete intersection pattern between $\mathcal{B}_1$ and $\mathcal{B}_2$ forms the grid $[n_1] \times [n_2]$, since $(\cB_1, \cB_2, \cB_3)$ is assumed to be restricted. 
    
    We now focus on the intersection of the configuration with the planes $\bm{\alpha}_k \eqdef \{y = x + k\}$,
    where $-(n_1-1) \le k \le (n_2-1)$. \hyperlink{figtwo}{Figure~2} illustrates this process for $\mathcal{B}_1$, $\mathcal{B}_2$, and $\mathcal{B}_3$, consisting of $6$ red, $4$ blue, and $2$ green rectangles, respectively. \hyperlink{figtwoA}{Figure~2.A} and \hyperlink{figtwoC}{Figure~2.C} depict the rectangles before and after rescaling, viewed from the top, respectively. 

    \begin{center}
    \begin{tikzpicture}[scale = 0.8]
        \clip (-1, -1) rectangle (8.2, 5.5);
        \filldraw[color = mgreen, opacity = 0.25] (2.22, 1.64) rectangle (5.28, 4.1);
        \filldraw[color = mgreen, opacity = 0.5] (3.66, 2.08) rectangle (4.17, 2.8);
        \foreach \x in {1, 1.5, 3.3, 3.9, 4.8, 6}
            \draw[dash pattern = on 1pt off 2pt] (\x, 0) -- (\x, 5);
        \foreach \y in {1, 1.8, 2.5, 4}
            \draw[dash pattern = on 1pt off 2pt] (0, \y) -- (7, \y);
        \draw[thick, color = mred] (1, 0.4) -- (1, 4.3);
        \draw[thick, color = mred] (1.5, 0.7) -- (1.5, 4.6);
        \draw[thick, color = mred] (3.3, 0.9) -- (3.3, 4.5);
        \draw[thick, color = mred] (3.9, 0.3) -- (3.9, 4.8);
        \draw[thick, color = mred] (4.8, 0.2) -- (4.8, 4.4);
        \draw[thick, color = mred] (6, 0.8) -- (6, 4.1);
        \draw[thick, color = mblue] (0.3, 1) -- (6.6, 1);
        \draw[thick, color = mblue] (0.6, 1.8) -- (6.2, 1.8);
        \draw[thick, color = mblue] (0.5, 2.5) -- (6.8, 2.5);
        \draw[thick, color = mblue] (0.8, 4) -- (6.3, 4);
        \draw (0.7, 3.7) -- (1.3, 4.3);
        \draw (0.7, 2.2) -- (1, 2.5) -- (1.5, 4) -- (1.8, 4.3);
        \draw (0.7, 1.5) -- (1, 1.8) -- (1.5, 2.5) -- (3.3, 4) -- (3.6, 4.3);
        \draw (0.7, 0.7) -- (1, 1) -- (1.5, 1.8) -- (3.3, 2.5) -- (3.9, 4) -- (4.2, 4.3);
        \draw[ultra thick] (1.2, 0.7) -- (1.5, 1) -- (3.3, 1.8) -- (3.9, 2.5) -- (4.8, 4) -- (5.1, 4.3);
        \draw (3, 0.7) -- (3.3, 1) -- (3.9, 1.8) -- (4.8, 2.5) -- (6, 4) -- (6.3, 4.3);
        \draw (3.6, 0.7) -- (3.9, 1) -- (4.8, 1.8) -- (6, 2.5) -- (6.3, 2.8);
        \draw (4.5, 0.7) -- (4.8, 1) -- (6, 1.8) -- (6.3, 2.1);
        \draw (5.7, 0.7) -- (6.3, 1.3);
        \foreach \x in {1, 1.5, 3.3, 3.9, 4.8, 6}
            \foreach \y in {1, 1.8, 2.5, 4}
                \filldraw[color = mmagenta] (\x, \y) circle (0.05);
        \node at (1.25, 0.45) {\footnotesize $\widetilde{\bm{\alpha}}$};
        \draw[->] (-0.3, -0.3) -- (-0.3, 0.5);
        \draw[->] (-0.3, -0.3) -- (0.5, -0.3);
        \node at (-0.3, 0.7) {\footnotesize $y$};
        \node at (0.7, -0.3) {\footnotesize $x$};
        \node at (3.6, -0.75) {\footnotesize \textbf{\hypertarget{figtwoA}{Figure 2.A}}};
    \end{tikzpicture}
    \begin{tikzpicture}[scale = 0.8]
    
        \clip (-1, -1) rectangle (8.2, 5.5);
    
        \pgfmathsetmacro{\xone}{1.396694}
        \pgfmathsetmacro{\xtwo}{3.161233}
        \pgfmathsetmacro{\xthree}{4.044299}
        \pgfmathsetmacro{\xfour}{5.673643}
    
        \def\H{4} 
    
        \pgfmathsetmacro{\ytopone}{4.7}
        \pgfmathsetmacro{\ytoptwo}{4.0214847855139135}
        \pgfmathsetmacro{\ytopthree}{4.089043998446608}
        \pgfmathsetmacro{\ytopfour}{4.496344988714392}
    
        \foreach \X in {\xone,\xtwo,\xthree,\xfour}
            \draw[dash pattern = on 1pt off 2pt] (\X, 0) -- (\X, 5);
    
        %
        %
        %
    
        \draw[thick, color = mred]   (\xone,   {(\ytopone-\H)+0.8})       -- (\xone,   {\ytopone-0.7});
        \draw[thick, color = mred]   (\xtwo,   {(\ytoptwo-\H)+1.2})       -- (\xtwo,   {\ytoptwo-2.0});
        \draw[thick, color = mred]   (\xthree, {(\ytopthree-\H)+1.9})     -- (\xthree, {\ytopthree-0.9});
        \draw[thick, color = mred]   (\xfour,  {(\ytopfour-\H)+1.0})      -- (\xfour,  {\ytopfour-1.2});
    
        \draw[thick, color = mblue]  (\xone,   {(\ytopone-\H)+1.1})       -- (\xone,   {\ytopone-1.3});
        \draw[thick, color = mblue]  (\xtwo,   {(\ytoptwo-\H)+0.6})       -- (\xtwo,   {\ytoptwo-0.6});
        \draw[thick, color = mblue]  (\xthree, {(\ytopthree-\H)+1.5})     -- (\xthree, {\ytopthree-1.8});
        \draw[thick, color = mblue]  (\xfour,  {(\ytopfour-\H)+1.3})      -- (\xfour,  {\ytopfour-0.8});
    
        \draw[ultra thick, color = mmagenta] (\xone,   {(\ytopone-\H)+1.1})   -- (\xone,   {\ytopone-1.3});
        \draw[ultra thick, color = mmagenta] (\xtwo,   {(\ytoptwo-\H)+1.2})   -- (\xtwo,   {\ytoptwo-2.0});
        \draw[ultra thick, color = mmagenta] (\xthree,   {(\ytopthree-\H)+1.9}) -- (\xthree,   {\ytopthree-1.8});
        \draw[ultra thick, color = mmagenta] (\xfour,  {(\ytopfour-\H)+1.3})  -- (\xfour,  {\ytopfour-1.2});
    
        \def\s{0.96}
        \coordinate (T0) at (0.7*1.414, 4.7);
    
        \path
          (T0) ++(0:\s*0.424264) coordinate (T1)
               ++(-21.037511:\s*1.969772) coordinate (T2)
               ++(4.398705:\s*0.921954) coordinate (T3)
               ++(14.036243:\s*1.749286) coordinate (T4)
               ++(0:\s*0.424264) coordinate (T5);
    
        \coordinate (B0) at ($(T0)+(0,-\H)$);
        \coordinate (B1) at ($(T1)+(0,-\H)$);
        \coordinate (B2) at ($(T2)+(0,-\H)$);
        \coordinate (B3) at ($(T3)+(0,-\H)$);
        \coordinate (B4) at ($(T4)+(0,-\H)$);
        \coordinate (B5) at ($(T5)+(0,-\H)$);
    
        \draw[ultra thick]
          (T0) -- (T1) -- (T2) -- (T3) -- (T4) -- (T5)
          -- (B5)
          -- (B4) -- (B3) -- (B2) -- (B1) -- (B0)
          -- cycle;
    
    
    
        \draw[thick, color=mgreen]
        (2.808325,1.550398)
        -- (3.161233,1.418977)
        -- (4.044299,1.489688)
        -- (5.673643,1.90)
        -- (5.815043,1.90);
        
        \filldraw[color = mgray] (3.161233,1.418977) circle (0.05);
        \filldraw[color = mgray] (5.673643,1.903952) circle (0.05);
        
        \draw[thick, color=mgreen]
        (3.686031,2.052860)
        -- (4.044299,2.081144)
        -- (4.393138,2.185997);
    
        \filldraw[color = mgray] (4.044299,2.081144) circle (0.05);
    
        \node at (6.27, 4.7) {\footnotesize $\widetilde{\bm{\alpha}}$};
    
        \draw[->] (-0.3, -0.3) -- (-0.3, 0.5);
        \draw[->] (-0.3, -0.3) -- (0.2, -0.5);
        \draw[->] (-0.3, -0.3) -- (0.2, -0.1);
        \node at (-0.3, 0.7) {\footnotesize $z$};
        \node at (0.35, -0.55) {\scriptsize $x$};
        \node at (0.35, -0.05) {\scriptsize $y$};
        \node at (3.6, -0.75) {\footnotesize \textbf{\hypertarget{figtwoB}{Figure 2.B}}};
    
    \end{tikzpicture}
    
    \vspace{1em}
    
    \begin{tikzpicture}[scale = 0.8]
        \clip (-1, -1) rectangle (8.2, 5.5);
        \filldraw[color = mgreen, opacity = 0.25] (2.4, 1.8) rectangle (5.4, 4.1);
        \filldraw[color = mgreen, opacity = 0.5] (3.6, 2.4) rectangle (4.3, 3.2);
        \foreach \x in {1, 2, 3, 4, 5, 6}
            \draw[dash pattern = on 1pt off 2pt] (\x, 0) -- (\x, 5);
        \foreach \y in {1, 2, 3, 4}
            \draw[dash pattern = on 1pt off 2pt] (0, \y) -- (7, \y);
        \draw[thick, color = mred] (1, 0.4) -- (1, 4.3);
        \draw[thick, color = mred] (2, 0.7) -- (2, 4.6);
        \draw[thick, color = mred] (3, 0.9) -- (3, 4.5);
        \draw[thick, color = mred] (4, 0.3) -- (4, 4.8);
        \draw[thick, color = mred] (5, 0.2) -- (5, 4.4);
        \draw[thick, color = mred] (6, 0.8) -- (6, 4.1);
        \draw[thick, color = mblue] (0.3, 1) -- (6.6, 1);
        \draw[thick, color = mblue] (0.6, 2) -- (6.2, 2);
        \draw[thick, color = mblue] (0.5, 3) -- (6.8, 3);
        \draw[thick, color = mblue] (0.8, 4) -- (6.3, 4);
        \draw (0.7, 3.7) -- (1.3, 4.3);
        \draw (0.7, 2.7) -- (2.3, 4.3);
        \draw (0.7, 1.7) -- (3.3, 4.3);
        \draw (0.7, 0.7) -- (4.3, 4.3);
        \draw[ultra thick] (1.7, 0.7) -- (5.3, 4.3);
        \draw (2.7, 0.7) -- (6.3, 4.3);
        \draw (3.7, 0.7) -- (6.3, 3.3);
        \draw (4.7, 0.7) -- (6.3, 2.3);
        \draw (5.7, 0.7) -- (6.3, 1.3);
        \foreach \x in {1, 2, 3, 4, 5, 6}
            \foreach \y in {1, 2, 3, 4}
                \filldraw[color = mmagenta] (\x, \y) circle (0.05);
        \node at (1, 5.2) {\tiny $x = 1$};
        \node at (2, 5.2) {\tiny $x = 2$};
        \node at (3, 5.2) {\tiny $x = 3$};
        \node at (4, 5.2) {\tiny $x = 4$};
        \node at (5, 5.2) {\tiny $x = 5$};
        \node at (6, 5.2) {\tiny $x = 6$};
        \node at (7.6, 1) {\tiny $y = 1$};
        \node at (7.6, 2) {\tiny $y = 2$};
        \node at (7.6, 3) {\tiny $y = 3$};
        \node at (7.6, 4) {\tiny $y = 4$};
        \node at (0.5, 3.5) {\footnotesize $\bm{\alpha}_3$};
        \node at (0.5, 2.5) {\footnotesize $\bm{\alpha}_2$};
        \node at (0.5, 1.5) {\footnotesize $\bm{\alpha}_1$};
        \node at (0.5, 0.5) {\footnotesize $\bm{\alpha}_0$};
        \node at (1.5, 0.5) {\footnotesize $\bm{\alpha}_{-1}$};
        \node at (2.5, 0.5) {\footnotesize $\bm{\alpha}_{-2}$};
        \node at (3.5, 0.5) {\footnotesize $\bm{\alpha}_{-3}$};
        \node at (4.5, 0.5) {\footnotesize $\bm{\alpha}_{-4}$};
        \node at (5.5, 0.5) {\footnotesize $\bm{\alpha}_{-5}$};
        \draw[->] (-0.3, -0.3) -- (-0.3, 0.5);
        \draw[->] (-0.3, -0.3) -- (0.5, -0.3);
        \node at (-0.3, 0.7) {\footnotesize $y$};
        \node at (0.7, -0.3) {\footnotesize $x$};
        \node at (3.6, -0.75) {\footnotesize \textbf{\hypertarget{figtwoC}{Figure 2.C}}};
    \end{tikzpicture}
    \begin{tikzpicture}[scale = 0.8]
        \clip (-1, -1) rectangle (8.2, 5.5);
        \foreach \x in {1, 2, 3, 4}
            \draw[dash pattern = on 1pt off 2pt] (1.414*\x, 0) -- (1.414*\x, 5);
        \draw[thick, color = mred] (1*1.414, 1.3) -- (1*1.414, 3.8);
        \draw[thick, color = mred] (2*1.414, 1.7) -- (2*1.414, 2.5);
        \draw[thick, color = mred] (3*1.414, 2.4) -- (3*1.414, 3.6);
        \draw[thick, color = mred] (4*1.414, 1.5) -- (4*1.414, 3.3);
        \draw[thick, color = mblue] (1*1.414, 1.6) -- (1*1.414, 3.2);
        \draw[thick, color = mblue] (2*1.414, 1.1) -- (2*1.414, 3.9);
        \draw[thick, color = mblue] (3*1.414, 2.0) -- (3*1.414, 2.7);
        \draw[thick, color = mblue] (4*1.414, 1.8) -- (4*1.414, 3.7);
        \draw[ultra thick, color = mmagenta] (1*1.414, 1.6) -- (1*1.414, 3.2);
        \draw[ultra thick, color = mmagenta] (2*1.414, 1.7) -- (2*1.414, 2.5);
        \draw[ultra thick, color = mmagenta] (3*1.414, 2.4) -- (3*1.414, 2.7);
        \draw[ultra thick, color = mmagenta] (4*1.414, 1.8) -- (4*1.414, 3.3);
        \draw[thick, color = mgreen] (2.6*1.414, 2.5) -- (3.2*1.414, 2.5);
        \draw[thick, color = mgreen] (1.8*1.414, 1.9) -- (4.1*1.414, 1.9);
        \filldraw[color = mgray] (3*1.414, 2.5) circle (0.05);
        \filldraw[color = mgray] (2*1.414, 1.9) circle (0.05);
        \filldraw[color = mgray] (4*1.414, 1.9) circle (0.05);
        \draw[ultra thick] (0.7*1.414, 0.5) rectangle (4.3*1.414, 4.5);
        \node at (6.5, 4.6) {\footnotesize $\bm{\alpha}_{-1}$};
        \node at (1*1.414, 5.2) {\tiny $(x, y) = (2, 1)$};
        \node at (2*1.414, -0.2) {\tiny $(x, y) = (3, 2)$};
        \node at (3*1.414, 5.2) {\tiny $(x, y) = (4, 3)$};
        \node at (4*1.414, -0.2) {\tiny $(x, y) = (5, 4)$};
        \draw[->] (-0.3, -0.3) -- (-0.3, 0.5);
        \draw[->] (-0.3, -0.3) -- (0.2, -0.5);
        \draw[->] (-0.3, -0.3) -- (0.2, -0.1);
        \node at (-0.3, 0.7) {\footnotesize $z$};
        \node at (0.35, -0.55) {\scriptsize $x$};
        \node at (0.35, -0.05) {\scriptsize $y$};
        \node at (3.6, -0.75) {\footnotesize \textbf{\hypertarget{figtwoD}{Figure 2.D}}};
    \end{tikzpicture}
    \begin{tikzpicture}
        \clip (-5, -0.25) rectangle (5, 0.5);
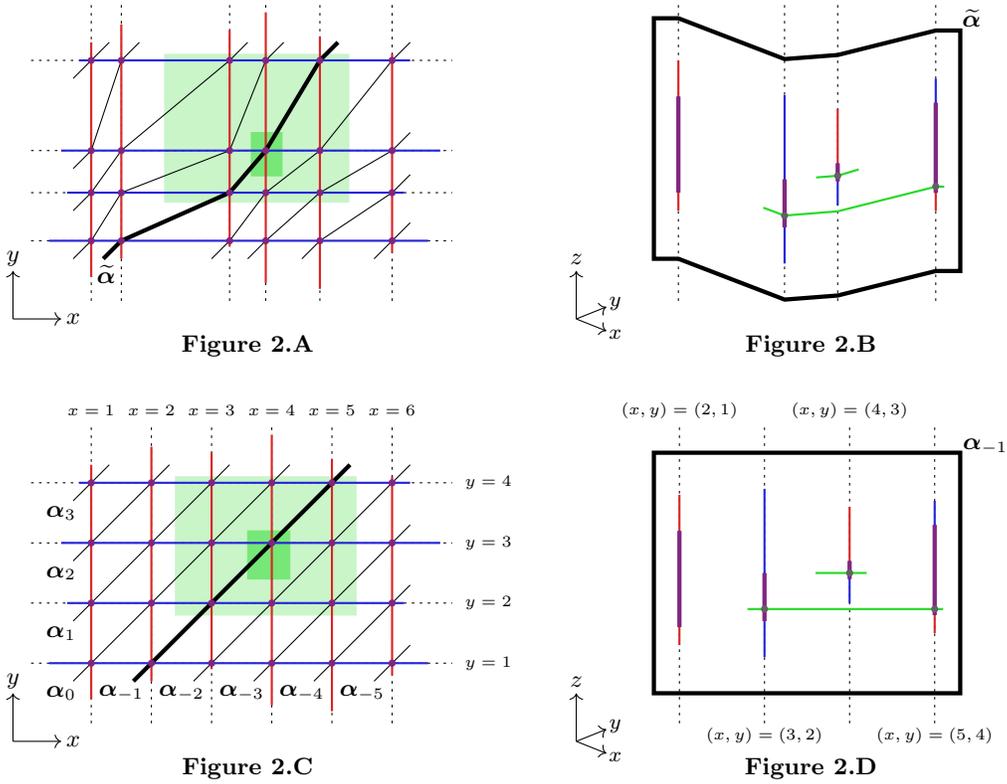
        \node at (0, 0) {\textbf{\hypertarget{figtwo}{Figure 2.}} An illustration of the rescaling.};
    \end{tikzpicture}
    \end{center}
    As shown in \hyperlink{figtwoC}{Figure~2.C}, we partition the intersections into parallel diagonal planes, corresponding, prior to rescaling, to piecewise linear surfaces (see \hyperlink{figtwoA}{Figure~2.A}). \hyperlink{figtwoD}{Figure~2.D} depicts one such plane, $\bm{\alpha}_{-1}$, corresponding to $\widetilde{\bm{\alpha}}$ in \hyperlink{figtwoB}{Figure~2.B}. Before rescaling, the green rectangles intersect $\widetilde{\bm{\alpha}}$ in parallel piecewise linear segments. We remark that throughout \hyperlink{figtwo}{Figure~2} each red-blue intersection is colored purple, while each triple intersection of red, blue, and green is colored gray. 
    
    Consider the sets of intersections
    \[
    \cS_k \eqdef \{\bm{b} \cap \bm{\alpha}_k : \bm{b} \in \cB_3\}, \qquad \cT_k \eqdef \bigl\{ \bm{b}_1^s \cap \bm{b}_2^{s+k} \cap \bm{\alpha}_k : s \in [n_1], \, s + k \in [n_2] \bigr\}. 
    \]
    Then $\cS_k, \cT_k$ are two families of parallel segments on $\bm{\alpha}_k$, where any pair of segments $\bm{m} \in \cS_k, \, \bm{n} \in \cT_k$ are perpendicular to each other. We also have $|\cS_k| \le n_3$ and $|\cT_{1-n_1}| + \dots + |\cT_{n_2-1}| = n_1n_2$. 

    Fix an arbitrary $k \in \{1 - n_1, \dots, n_2 - 1\}$. We claim that the intersection graph $G_k$ between $\cS_k$ and $\cT_k$ is $K_{t, t}$-free. Suppose to the contrary that $\{\bm{m}_1, \dots, \bm{m}_t\} \subseteq \cS_k$ and $\{\bm{n}_1, \dots, \bm{n}_t\} \subseteq \cT_k$ induce a $K_{t, t}$ in $G_k$, where $\bm{m}_j = \bm{b}_3^j \cap \bm{\alpha}_k$ for some $\bm{b}_3^j \in \cB_3$ and $\bm{n}_j = \bm{b}_1^{s_j} \cap \bm{b}_2^{s_j+k} \cap \bm{\alpha}_k$. Consider
    \[
    \bigl\{ \bm{b}_1^{s_1}, \dots, \bm{b}_1^{s_t} \bigr\} \subseteq \cB_1, \qquad \bigl\{ \bm{b}_2^{s_1+k}, \dots, \bm{b}_2^{s_t+k} \bigr\} \subseteq \cB_2, \qquad \bigl\{ \bm{b}_3^1, \dots, \bm{b}_3^t \bigr\} \subseteq \cB_3. 
    \]
    Each pair of the above boxes from distinct families intersects. Since our boxes have Helly number $2$ (see \Cref{obs:box_helly}), the above $3t$ boxes induce a $K_{t, t, t}^{(3)}$-subgraph in $\cH$, a contradiction. 
    
    It thus follows from \Cref{lem:COZ_CKS} that $e(G_k) \le 27t(|\cS_k| + |\cT_k|)$. By summing over $k$, we obtain
    \[
    e(\cH) = \sum_{k=1-n_1}^{n_2-1} e(G_k) \le \sum_{k=1-n_1}^{n_2-1} 27t(|\cS_k| + |\cT_k|) \le 27t \bigl( n_3(n_1+n_2) + n_1n_2 \bigr) = 27g_t(n_1, n_2, n_3). \qedhere
    \]
\end{proof}

\begin{proof}[Proof of \Cref{prop:upper_restricted}]
    Let $(\cB_1, \dots, \cB_d)$ be a restricted family, where $|\cB_i| = n_i$ for $i = 1, \dots, d$, and denote by $\cH$ the intersection hypergraph of $(\cB_1, \dots, \cB_d)$. We will show that $e(\cH) \le 27g_t(n_1, \dots, n_d)$. 

    In the following, we will use $x_i$ to denote the $i$-th coordinate of $\mathbb{R}^d$. Similar to the $d = 3$ case, since $(\cB_1, \dots, \cB_d)$ is separated, we may assume without loss of generality that the following holds. 
    \vspace{-0.5em}
    \begin{itemize}
        \item For $i = 1, \dots, d - 1$, the hyperplane $\{x_i = u\}$ contains exactly one box from $\cB_i$ for all $u \in [n_i]$. 
    \end{itemize}
    \vspace{-0.5em}

    Let $\vec{e}_i$ be the unit vector in the $i$-th coordinate direction of $\R^d$. For instance, $\vec{e}_1 = (1, 0, \dots, 0)$. Consider the ``lower-half'' boundary of the intersection grid between $\cB_1, \dots, \cB_{d-1}$, which is
    \begin{align*}
        P &\eqdef ([n_1] \times \dots \times [n_{d-1}] \times \{0\}) \cap (\{x_1 = 1\} \cup \dots \cup \{x_{d-1} = 1\}) \\
        &= \bigl\{ (x_1, \dots, x_{d-1}, 0) \in [n_1] \times \dots \times [n_{d-1}] \times \{0\} : (x_1-1) \cdots (x_{d-1}-1) = 0 \bigr\}. 
    \end{align*}
    It follows that $|P| \le \frac{n_1 \cdots n_{d-1}}{n_1} + \dots + \frac{n_1 \cdots n_{d-1}}{n_{d-1}} = \frac{g_t(n_1, \dots, n_{d-1})}{t}$. For each point $p \in P$, let $\bm{\alpha}_p$ be the 2-dimensional plane through $p$ that is spanned by vectors $\vec{e}_1 + \dots + \vec{e}_{d-1}$ and $\vec{e}_d$. Observe that every intersection between $\cB_1, \dots, \cB_d$ lies in exactly one of the planes $\bm{\alpha}_p$. For $p = (p_1, \dots, p_{d-1}, 0)$, define
    \[
    \cS_p \eqdef \{\bm{b} \cap \bm{\alpha}_p : \bm{b} \in \cB_d\}, \qquad \cT_p \eqdef \bigl\{ \bm{b}_1^{p_1+s} \cap \dots \cap \bm{b}_{d-1}^{p_{d-1}+s} \cap \bm{\alpha}_p : p_i + s \in [n_i] \, (\forall i \in [d-1]) \bigr\}. 
    \]
    Then $\cS_p, \cT_p$ are two families of parallel segments on $\bm{\alpha}_p$, where any pair of segments $\bm{m} \in \cS_p, \, \bm{n} \in \cT_p$ are perpendicular to each other. We also have $|\cS_p| \le n_d$ and $\sum_{p \in P} |\cT_p| = n_1 \cdots n_{d-1}$. 

    Fix any $p \in P$. We claim that the intersection graph $G_p$ between $\cS_p$ and $\cT_p$ is $K_{t, t}$-free. Suppose for the sake of contradiction that segments $\{\bm{m}_1, \dots, \bm{m}_t\} \subseteq \cS_p$ and $\{\bm{n}_1, \dots, \bm{n}_t\} \subseteq \cT_p$ induce a $K_{t, t}$ in $G_p$, where $\bm{m}_j = \bm{b}_d^j \cap \bm{\alpha}_p$ for some $\bm{b}_d^j \in \cB_d$ and $\bm{n}_j = \bm{b}_1^{p_1+s_j} \cap \dots \cap \bm{b}_{d-1}^{p_{d-1}+s_j} \cap \bm{\alpha}_p$. Consider
    \[
    \bigl\{ \bm{b}_1^{p_1+s_1}, \dots, \bm{b}_1^{p_1+s_t} \bigr\} \subseteq \cB_1, \quad \dots, \quad \bigl\{ \bm{b}_{d-1}^{p_{d-1}+s_1}, \dots, \bm{b}_{d-1}^{p_{d-1}+s_t} \bigr\} \subseteq \cB_{d-1}, \quad \bigl\{ \bm{b}_d^1, \dots, \bm{b}_d^t \bigr\} \subseteq \cB_d. 
    \]
    Each pair of the above boxes from distinct families intersects. Since our boxes have Helly number $2$ (see \Cref{obs:box_helly}), the above $dt$ boxes induce a $K_{t, \dots, t}^{(d)}$-subgraph in $\cH$, a contradiction. 
    
    It thus follows from \Cref{lem:COZ_CKS} that $e(G_p) \le 27t(|\cS_p| + |\cT_p|)$. By summing over $p$, we obtain
    \begin{align*}
        e(\cH) = \sum_{p \in P} e(G_p) &\le \sum_{p \in P} 27t(|\cS_p| + |\cT_p|) \le 27t( n_d \cdot |P| + n_1 \cdots n_{d-1} ) \\
        &\le 27t \bigl( n_d \cdot \tfrac{g_t(n_1, \dots, n_{d-1})}{t} + n_1 \cdots n_{d-1} \bigr) = 27g_t(n_1, \dots, n_d). \qedhere
    \end{align*}
\end{proof}

\subsection{From the restricted case to the separated case} \label{sec:upper_separated}

Secondly, we prove \Cref{thm:upper_general} in the separated case. 

\begin{proposition} \label{prop:upper_separated}
    Suppose $d \ge 2$ and $t \ge 2$. There exists $M_d > 0$, a constant depending only on $d$, such that $z_d^*(\vn; t) \le M_d \cdot g_t(\vn)$ holds for every $d$-size-vector $\vn$. 
\end{proposition}

To deduce \Cref{prop:upper_separated}, we follow the Chan--Keller--Smorodinsky argument~\cite[Section 2.2.2]{chan_keller_smorodinsky} closely, beginning with their efficient biclique cover result. A \emph{biclique} is a complete bipartite graph $K_{m, n}$ of \emph{size} $m \times n$. We refer to a subgraph of the biclique $K_{m, n}$ as a \emph{partial} biclique of size $m \times n$. Here we say that ``graph $G$ is covered by $H_1, \dots, H_k$'' if ``there exist vertex sets $V_1, \dots, V_k$ such that $H_i = G[V_i]$ is an \textbf{induced} subgraph for any $i \in [k]$ and $E(G) = E(H_1) \cup \dots \cup E(H_k)$''. We shall use the same ``cover'' terminology and definition for hypergraphs as well. 

\begin{lemma}[{\cite[Lemma 11]{chan_keller_smorodinsky}}] \label{lem:biclique_cover}
    Suppose $m, n, b$ are positive integers with $b \le \min\{m, n\}$. Let $(\cA, \cB)$ be a $\bigl( [d],[d] \bigr)$-family of boxes in $\R^d$ with $|\cA| = m, \, |\cB| = n$. Assume further that $\cA \cup \cB$ is in general position, in the sense that no pair of $\ell$-dimensional facets belonging to distinct boxes and having the same directions lie on a common $\ell$-flat. Let $G$ be the intersection graph between $\cA$ and $\cB$. Then there exists a constant $C_d > 0$, depending only on $d$, such that $G$ can be covered\footnote{In the original statement of their paper, this is formulated as a partition of \(G\). However, in applications the bicliques and partial bicliques must preserve the intersection graph structure. In fact, the divide-and-conquer proof of the lemma itself shows that these bicliques and partial bicliques can be chosen as induced subgraphs that cover \(G\).} by a union of
    \vspace{-0.5em}
    \begin{itemize}
        \item at most $C_d \cdot b \log^{d-1} b$ many bicliques of arbitrary sizes, and 
        \vspace{-0.5em}
        \item at most $C_d \cdot b \log^{d-1} b$ many partial bicliques of sizes $\frac{m}{b} \times \frac{n}{b}$. 
    \end{itemize}
    \vspace{-0.5em}
\end{lemma}

\Cref{lem:biclique_cover} is applicable to any pair of $\cB_i$ and $\cB_j$ in every separated $\vF_d$-family $(\cB_1, \dots, \cB_d)$. To see this, we consider the following perturbation. Enlarge each box $\bm{b} \in \bigcup_{i=1}^d \cB_i$ as 
\[
\bm{b} \eqdef [a_1, b_1] \times \dots \times [a_d, b_d] \quad \mapsto \quad \bm{b}^{\veps} \eqdef [a_1-\veps, b_1+\veps] \times \dots \times [a_d-\veps, b_d+\veps]. 
\]
For each $\bm{b} \in \bigcup_{i=1}^d \cB_i$, we pick a distinct sufficiently small $\veps_{\bm{b}} > 0$. Let $(\widetilde{\cB}_1, \dots, \widetilde{\cB}_d)$ be the family of boxes obtained by replacing every $\bm{b}$ with $\bm{b}^{\veps_{\bm{b}}}$. In this way, we know that $(\widetilde{\cB}_i, \widetilde{\cB}_j)$ forms a $\bigl( [d], [d] \bigr)$-family. Since all boxes are closed and any perturbation constant $\veps_{\bm{b}}$ is small enough, the perturbed family $(\widetilde{\cB}_i, \widetilde{\cB}_j)$ has exactly the same intersection graph as $(\cB_i, \cB_j)$ for all $i\neq j$. Moreover, after this perturbation, the collection of all boxes $\bigcup_{i=1}^d \widetilde{\cB}_i$ is in general position. It follows that \Cref{lem:biclique_cover} applies to $\widetilde{\cB}_i$ and $\widetilde{\cB}_j$, and hence it applies to the original $\cB_i$ and $\cB_j$ as well. We shall always use \Cref{lem:biclique_cover} in the above manner in the proof of \Cref{prop:upper_separated} below. 

\medskip
To show that an arbitrary separated family behaves similarly to a restricted separated family, we characterize such families by auxiliary graphs $G$ on the vertex set $[d]$, placing an edge between $i$ and $j$ if and only if $\cB_i$ and $\cB_j$ do not fully intersect. In this way, restricted families correspond to star graphs. When $G$ is not a star, our strategy is to estimate the maximum number of incidences recursively in terms of $e(G)$, with the help of efficient biclique covers (\Cref{lem:biclique_cover}). 

\begin{proof}[Proof of \Cref{prop:upper_separated}]
    Suppose $d, t \ge 2$. Recall that $z_d^*(n_1, \dots, n_d; t)$ denotes the maximum of $e(\cH)$ over all intersection hypergraphs $\cH$ of separated $\vF_d$-families $(\cB_1, \dots, \cB_d)$, where $|\cB_i| = n_i$ for $i \in [d]$, that forbid $K_{t, \dots, t}^{(d)}$. For each such $\cH$, we define an auxiliary graph $G$ on vertex set $[d]$, where $i$ and $j$ are adjacent if and only if $\cB_i$ and $\cB_j$ are incomplete (i.e., not every pair $\bm{b}_i \in \cB_i, \, \bm{b}_j \in \cB_j$ intersects). Let $z_G^*(\vn; t)$ denote the maximum number of edges in $\cH$ over all such hypergraphs whose auxiliary graph is a subgraph of $G$ on the same vertex set $[d]$ as $G$. 
    
    Abbreviate $u \eqdef e(G)$. Our goal is to find a constant $M_{u, d} > 0$ such that $z^*_G(\vn; t) \le M_{u, d} \cdot g_t(\vn)$ holds for every $d$-size-vector $\vn$. Once this has been done, \Cref{prop:upper_separated} follows easily upon setting $M_d \eqdef \max \bigl\{ M_{u, d} : 0 \le u \le \binom{d}{2} \bigr\}$. It therefore remains to construct $M_{u, d}$ by induction on $u$. 
    
    We first check the inductive base. If $u = 0$, then $G$ is empty graph. For all $1 \le i < j \le d$, any $\bm{b}_i \in \cB_i$ intersects each $\bm{b}_j \in \cB_j$. Since boxes have Helly number $2$ (\Cref{obs:box_helly}), the hypergraph $\cH$ is complete and $K_{t, \dots, t}^{(d)}$-free, and so there must be some $k \in [d]$ satisfying $n_k < t$. We thus obtain
    \[
    z^*_G(\vn; t) < (n_1 \cdots n_{k-1}) \cdot t \cdot (n_{k+1} \cdots n_d) \le t \cdot \frac{n_1 \cdots n_d}{\min\{n_1, \dots, n_d\}} \le g_t(\vn). 
    \]
    Therefore, the inductive base holds for the constant $M_{0, d} = 1$. 

    Suppose $e(G) = u > 0$ then. The inductive hypothesis shows that there is constant $M_{u-1, d} > 0$ such that $z^*_{G-e}(\vn; t) \le M_{u-1, d} \cdot g_t(\vn)$ holds for any $e \in E(G)$ and any $\vn$. Choose a constant $b > 10$ depending only on $d$ to be sufficiently large so that the followings hold: 
    \begin{equation} \label{eq:b_def}
        \max \Bigl\{ C_d^2 \cdot 2^{d-1} \cdot \tfrac{\log^{2d-2}b}{b}, \, C_d^3 \cdot 2^{d-1} \cdot \tfrac{\log^{3d-3}b}{b} \Bigr\} \le \tfrac{1}{2}. 
    \end{equation}
    Here $C_d > 0$ is the constant given by \Cref{lem:biclique_cover}. We also write $B_d \eqdef C_d \cdot b\log^{d-1} b$ and define
    \begin{equation} \label{eq:M_def}
        M_{u, d} \eqdef \max \Bigl\{ \tfrac{b^2}{t}, \, 27, \, 2(B_d+B_d^2+B_d^3) \cdot M_{u-1,d} \Bigr\}. 
    \end{equation}
    
    We establish $z^*_G(\vn; t) \le M_{u, d} \cdot g_t(\vn)$ via induction on $\vn$, where \Cref{lem:biclique_cover} is applied with this choice of $b$ in the inductive step. Once this is done, the inductive construction of $M_{u, d}$ on $u$ follows. 
    
    For the base, if there exists an entry $n_i$ of $\vn$ with $n_i \le b^2$, then
    \[
    z^*_G(\vn; t) \le (n_1 \cdots n_{i-1}) \cdot b^2 \cdot (n_{i+1} \cdots n_d) \overset{\eqref{eq:M_def}}{\le} (n_1 \cdots n_{i-1}) \cdot tM_{u, d} \cdot (n_{i+1} \cdots n_d) \le M_{u, d} \cdot g_t(\vn). 
    \]
    
    For the inductive step, we assume that every entry of $\vn$ is greater than $b^2$. We will only need the inductive hypothesis that $z^*_G(\lceil\frac{\vn}{b}\rceil; t)\le M_{u, d} \cdot g_t(\lceil\frac{\vn}{b}\rceil)$ and $z^*_G(\lceil\frac{\vn}{b^2}\rceil; t) \le M_{u, d} \cdot g_t(\lceil\frac{\vn}{b^2}\rceil)$ hold, where $\lceil\frac{\vn}{\lambda}\rceil \eqdef (\lceil\frac{n_1}{\lambda}\rceil, \dots, \lceil\frac{n_d}{\lambda}\rceil)$ for every $d$-size-vector $\vn = (n_1, \dots, n_d)$ and every $\lambda \in \{b, b^2\}$. Since $\lceil\tfrac{\vn}{\lambda}\rceil$ is entry-wise smaller than $\vn$ and $g_t(\lceil\tfrac{\vn}{\lambda}\rceil) \le g_t(\tfrac{2}{\lambda} \cdot \vn) = \frac{2^{d-1}}{\lambda^{d-1}} \cdot g_t(\vn)$, the inductive hypothesis shows that
    \begin{align}
        z^*_G(\lceil\tfrac{\vn}{b}\rceil; t) &\le M_{u, d} \cdot g_t(\lceil\tfrac{\vn}{b}\rceil) \le \tfrac{2^{d-1}}{b^{d-1}} \cdot M_{u, d} \cdot g_t(\vn), \label{eq:b_ceil} \\
        z^*_G(\lceil\tfrac{\vn}{b^2}\rceil; t) &\le M_{u, d} \cdot g_t(\lceil\tfrac{\vn}{b^2}\rceil) \le \tfrac{2^{d-1}}{b^{2d-2}} \cdot M_{u, d} \cdot g_t(\vn). \label{eq:b2_ceil} 
    \end{align}
    Concerning the structure of $G$ there are three cases: 
    \vspace{-0.5em}
    \begin{itemize}
        \item $G$ is a star (i.e., $G$ is a subgraph of $K_{1, d-1}$); 
        \vspace{-0.5em}
        \item $G$ contains two vertex-disjoint edges; 
        \vspace{-0.5em}
        \item $G$ contains a triangle. 
    \end{itemize}
    \vspace{-0.5em}

    \medskip
    \noindent\textbf{Case 1: $G$ is a star.}
    Then $G$ has an independent set of size $d - 1$, so some $d - 1$ of the $d$ sets $\cB_1, \dots, \cB_d$ intersect completely. Thus, the separated $\vF_d$-family is restricted, and \Cref{prop:upper_restricted} shows
    \[
    z^*_G(\vn; t) \le \wz_d^*(\vn; t) \le 27 g_t(\vn). 
    \]
    Due to the definition \eqref{eq:M_def}, we have $M_{u, d} \ge 27$, and so the inductive proof in this case is complete. 
    
    \medskip
    \noindent\textbf{Case 2: $G$ contains two vertex-disjoint edges.} Assume without loss of generality that $G$ has two edges $e_1 \eqdef \{1, 2\}$ and $e_2 \eqdef \{3, 4\}$. Our goal is to establish the recursive estimate
    \begin{equation} \label{eq:restricted_matching}
        z^*_G(\vn; t) \le B_d \cdot z^*_{G - e_1}(\vn; t) + B_d^2 \cdot z^*_{G - e_2}(\vn; t) + B_d^2 \cdot b^{d-4} \cdot z^*_G(\lceil\tfrac{\vn}{b}\rceil; t). 
    \end{equation} 
    Suppose \eqref{eq:restricted_matching} is true, then the inductive hypothesis on $u$ and \eqref{eq:b_ceil} show that
    \begin{align*}
        z^*_G(\vn; t) &\le B_d \cdot M_{u-1, d} \cdot g_t(\vn) + B_d^2 \cdot M_{u-1, d} \cdot g_t(\vn) + B_d^2 \cdot b^{d-4} \cdot \tfrac{2^{d-1}}{b^{d-1}} \cdot M_{u, d} \cdot g_t(\vn) \\
        &\le \bigl( (B_d+B_d^2) \cdot M_{u-1, d} + C_d^2  \cdot 2^{d-1} \cdot \tfrac{\log^{2d-2} b}{b} \cdot M_{u, d} \bigr) \cdot g_t(\vn) \\
        &\le \bigl( \tfrac{M_{u, d}}{2} + \tfrac{M_{u, d}}{2} \bigr) \cdot g_t(\vn) = M_{u, d} \cdot g_t(\vn), 
    \end{align*} 
    where the last ``$\le$'' follows from the constructions \eqref{eq:b_def} and \eqref{eq:M_def} of our parameters. 
    
    \smallskip
    
    It remains to prove \eqref{eq:restricted_matching}. Recall that $G$ has edges $e_1 = \{1, 2\}, \, e_2 = \{3, 4\}$ (i.e., the intersections between the first and second families, and between the third and fourth families, are incomplete). Suppose $\cH$ is the intersection hypergraph of separated $\vF_d$-family $(\cB_1, \dots, \cB_d)$, where $|\cB_i| = n_i$ for $i \in [d]$, forbidding $K_{t, \dots, t}^{(d)}$. Let $H'$ be the intersection graph of $(\cB_1, \cB_2)$. We may apply \Cref{lem:biclique_cover} to $H'$ and obtain a covering of $H'$ with at most $B_d$ many bicliques and $B_d$ many partial bicliques of sizes $\frac{n_1}{b} \times \frac{n_2}{b}$. Suppose the bicliques are $X'_1, \dots, X'_{\alpha}$ and the partial bicliques are $Y'_1, \dots, Y'_{\beta}$, where $\alpha, \beta \le B_d$. For $i = 1, \dots, \alpha$ and $j = 1, \dots, \beta$, extend subgraphs $X_i'$ and $Y_j'$ of $H'$ to subhypergraphs $\cX_i$ and $\cY_j$ of $\cH'$, where $\cX_i$ (resp.~$\cY_j$) refers to the intersection hypergraph of exactly boxes (vertices) from $X_i'$ (resp.~$Y_j'$) in $\cB_1 \cup \cB_2$ and every box in $\cB_3 \cup \dots \cup \cB_d$. By construction, 
    \vspace{-0.5em}
    \begin{itemize}
        \item the subhypergraphs $\cX_1, \dots, \cX_{\alpha}, \cY_1, \dots, \cY_{\beta}$ cover $\cH$, and 
        \vspace{-0.5em}
        \item the auxiliary graph of each $\cX_i$ is a subgraph of $G - e_1$, and 
        \vspace{-0.5em}
        \item the size-vector of each $\cY_j$ is entry-wise less than or equal to $(\lfloor \tfrac{n_1}{b} \rfloor, \lfloor \tfrac{n_2}{b} \rfloor, n_3, \dots, n_d)$. 
    \end{itemize}
    \vspace{-0.5em}
    It thus follows that
    \begin{equation} \label{eq:matching_est1}
        z^*_G(\vn; t) \le B_d \cdot z^*_{G-e_1}(\vn; t) + B_d \cdot z^*_G(\lfloor\tfrac{n_1}{b}\rfloor, \lfloor\tfrac{n_2}{b}\rfloor, n_3, \dots, n_d; t). 
    \end{equation}
    Similarly, if we apply \Cref{lem:biclique_cover} again to cover the intersection graph between the third and fourth families by $B_d$ bicliques and $B_d$ partial bicliques of size $\frac{n_1}{b} \times \frac{n_2}{b}$, then we obtain
    \begin{equation} \label{eq:matching_est2}
        z^*_G(\lfloor\tfrac{n_1}{b}\rfloor, \lfloor\tfrac{n_2}{b}\rfloor, n_3, \dots, n_d; t) \le B_d \cdot z^*_{G-e_2}(\vn; t) + B_d \cdot z^*_G(\lfloor\tfrac{n_1}{b}\rfloor, \lfloor\tfrac{n_2}{b}\rfloor,\lfloor\tfrac{n_3}{b}\rfloor,\lfloor\tfrac{n_4}{b}\rfloor, n_5, \dots, n_d; t). 
    \end{equation}
    By partitioning each of the fifth through the $d$-th families evenly into $b$ subfamilies, we deduce that
    \begin{align} 
        z^*_G(\lfloor\tfrac{n_1}{b}\rfloor, \lfloor\tfrac{n_2}{b}\rfloor,\lfloor\tfrac{n_3}{b}\rfloor,\lfloor\tfrac{n_4}{b}\rfloor, n_5, \dots, n_d; t) &\le b^{d-4} \cdot z^*_G(\lceil\tfrac{n_1}{b}\rceil, \lceil\tfrac{n_2}{b}\rceil,\lceil\tfrac{n_3}{b}\rceil,\lceil\tfrac{n_4}{b}\rceil, \lceil\tfrac{n_5}{b}\rceil, \dots, \lceil\tfrac{n_d}{b}\rceil; t)\nonumber\\
        &= b^{d-4} \cdot z^*_G(\lceil\tfrac{\vn}{b}\rceil; t). \label{eq:matching_est3}
    \end{align}
    Finally, \eqref{eq:restricted_matching} follows from a combination of \eqref{eq:matching_est1}, \eqref{eq:matching_est2}, and \eqref{eq:matching_est3}. 
    
    \medskip
    \noindent\textbf{Case 3: $G$ contains a triangle.} Assume without loss of generality that $e_1 \eqdef \{1, 2\}, \, e_2 \eqdef \{2, 3\}$, and $e_3 \eqdef \{3, 1\}$ form a triangle in $G$. Our goal is to establish the recursive estimate
    \begin{equation} \label{eq:restricted_triangle}
        z^*_G(\vn; t) \le B_d \cdot z^*_{G - e_1}(\vn; t) + B_d^2 \cdot z^*_{G - e_2}(\vn; t) + B_d^3 \cdot z^*_{G - e_3}(\vn; t) + B_d^3 \cdot b^{2d-6} \cdot z^*_G(\lceil\tfrac{\vn}{b^2}\rceil; t).
    \end{equation}
    Suppose \eqref{eq:restricted_triangle} is true, then the induction hypothesis on $u$ and \eqref{eq:b2_ceil} show that
    \begin{align*}
        z^*_G(\vn; t) &\le B_d \cdot M_{u-1, d} \cdot g_t(\vn) + B_d^2 \cdot M_{u-1, d} \cdot g_t(\vn) + B_d^3 \cdot M_{u-1, d} \cdot g_t(\vn) \\
        &\qquad + B_d^3 \cdot b^{2d-6} \cdot \tfrac{2^{d-1}}{b^{2d-2}} \cdot M_{u, d} \cdot g_t(\vn) \\
        &\le \bigl( (B_d+B_d^2+B_d^3) \cdot M_{u-1, d} + C_d^3 \cdot 2^{d-1} \cdot \tfrac{\log^{3d-3} b}{b} \cdot M_{u, d} \bigr) \cdot g_t(\vn) \\
        &\le \bigl( \tfrac{M_{u, d}}{2} + \tfrac{M_{u, d}}{2} \bigr) \cdot g_t(\vn) = M_{u, d} \cdot g_t(\vn), 
    \end{align*} 
    where the last ``$\le$'' follows from the constructions \eqref{eq:b_def} and \eqref{eq:M_def} of our parameters. 
    
    \smallskip
    
    It remains to prove \eqref{eq:restricted_triangle}. Recall that $G$ contains a triangle $e_1 = \{1, 2\}, \, e_2 = \{2, 3\}, \, e_3 = \{3, 1\}$. Similarly as in the previous analysis, by \Cref{lem:biclique_cover} applied to the intersection graph between the first and second families, the second and third families, and the third and first families, we obtain
    \begin{align}
        z^*_G(\vn; t) &\le B_d \cdot z^*_{G-e_1}(\vn; t) + B_d \cdot z^*_G( \lfloor \tfrac{n_1}{b} \rfloor, \lfloor \tfrac{n_2}{b} \rfloor, n_3, n_4, \dots, n_d; t), \label{eq:triangle_est1} \\
        z^*_G( \lfloor \tfrac{n_1}{b} \rfloor, \lfloor \tfrac{n_2}{b} \rfloor, n_3, n_4, \dots, n_d; t) &\le B_d \cdot z^*_{G-e_2}(\vn; t) + B_d \cdot z^*_G( \lfloor \tfrac{n_1}{b} \rfloor, \lfloor \tfrac{n_2}{b^2} \rfloor, \lfloor \tfrac{n_3}{b} \rfloor, n_4, \dots, n_d; t), \label{eq:triangle_est2} \\
        z^*_G( \lfloor \tfrac{n_1}{b} \rfloor, \lfloor \tfrac{n_2}{b^2} \rfloor, \lfloor \tfrac{n_3}{b} \rfloor, n_4, \dots, n_d; t) &\le B_d \cdot z^*_{G-e_3}(\vn; t) + B_d \cdot z^*_G( \lfloor \tfrac{n_1}{b^2} \rfloor, \lfloor \tfrac{n_2}{b^2} \rfloor, \lfloor \tfrac{n_3}{b^2} \rfloor, n_4, \dots, n_d; t). \label{eq:triangle_est3}
    \end{align}
    By partitioning each of the fourth through the $d$-th family evenly into $b^2$ subfamilies, we deduce that
    \begin{align}
        z^*_G( \lfloor \tfrac{n_1}{b^2} \rfloor, \lfloor \tfrac{n_2}{b^2} \rfloor, \lfloor \tfrac{n_3}{b^2}\rfloor, n_4, \dots, n_d; t) &\le (b^2)^{d-3} \cdot z^*_G( \lceil \tfrac{n_1}{b^2} \rceil, \lceil \tfrac{n_2}{b^2} \rceil, \lceil \tfrac{n_3}{b^2} \rceil, \lceil \tfrac{n_4}{b^2} \rceil, \dots, \lceil \tfrac{n_d}{b^2} \rceil; t) \nonumber \\
        &= b^{2d-6} \cdot z^*_G( \lceil \tfrac{\vn}{b^2} \rceil; t). 
        \label{eq:triangle_est4}
    \end{align} 
    Finally, \eqref{eq:restricted_triangle} follows from a combination of \eqref{eq:triangle_est1}, \eqref{eq:triangle_est2}, \eqref{eq:triangle_est3}, and \eqref{eq:triangle_est4}. 
    
    \medskip
    The induction on $\vn$ is done in each case, and hence the proof of \Cref{prop:upper_separated} is complete. 
\end{proof}

\subsection{From the separated case to the canonical case} \label{sec:upper_canonical}

Thirdly, we prove \Cref{thm:upper_general} in the canonical case. 

\begin{proposition} \label{prop:upper_canonical}
    Suppose $d \ge 2$ and $t \ge 2$. Then $z_d(\vn; t) = z_d^*(\vn; t) \le M_d \cdot g_t(\vn)$ holds for every $d$-size-vector $\vn$, where the constant $M_d > 0$ is given by \Cref{prop:upper_separated}.
    \vspace{-0.5em}
\end{proposition}

To prove \Cref{prop:upper_canonical}, it suffices to show that we can perturb the boxes while preserving the intersection hypergraph, so that the perturbed family becomes separated. To this end, we slightly enlarge each box to create some space for small perturbations of the hyperplanes. 

\begin{proof}[Proof of \Cref{prop:upper_canonical}]
    Fix an arbitrary canonical $\vF_d$-family $(\cB_1, \dots, \cB_d)$. For $\bm{b}_1 \in \cB_1, \dots, \bm{b}_d \in \cB_d$, if they intersect, then the $d$-direction-vector $\vF_d = \bigl( [d] \setminus \{1\}, \dots, [d] \setminus \{d\} \bigr)$ shows that the intersection must be a single point. For each $i \in [d]$ and any $\bm{b} \in \cB_i$, we consider the following enlargement of $\bm{b}$.
    \[
    \begin{array}{ccccccccccccccc}
        \bm{b} &\eqdef &{\scriptstyle [a_1, b_1]} &{\scriptstyle \times} &{\scriptstyle \cdots} &{\scriptstyle \times} &{\scriptstyle [a_{i-1}, b_{i-1}]} &{\scriptstyle \times} &{\scriptstyle \{a_i\}} &{\scriptstyle \times} &{\scriptstyle [a_{i+1}, b_{i+1}]} &{\scriptstyle \times} &{\scriptstyle \cdots} &{\scriptstyle \times} &{\scriptstyle [a_d, b_d]} \\[-1.25em]
        \rotatebox{-90}{$\mapsto$} & & & & & & & & & & & & & & \\[0.5em]
        \bm{b}^{\veps, i} &\eqdef &{\scriptstyle [a_1-\veps, b_1+\veps]} &{\scriptstyle \times} &{\scriptstyle \cdots} &{\scriptstyle \times} &{\scriptstyle [a_{i-1}-\veps, b_{i-1}+\veps]} &{\scriptstyle \times} &{\scriptstyle \{a_i\}} &{\scriptstyle \times} &{\scriptstyle [a_{i+1}-\veps, b_{i+1}+\veps]} &{\scriptstyle \times} &{\scriptstyle \cdots} &{\scriptstyle \times} &{\scriptstyle [a_d-\veps, b_d+\veps]}. 
    \end{array}
    \]
    Intersections are preserved under the enlargement. Moreover, if $\bm{b}_1 \in \cB_1, \dots, \bm{b}_d \in \cB_d$ do not intersect, then they still do not intersect after enlargement, given $\veps > 0$ is small enough. Thus, the intersection hypergraph of $(\cB_1^{\veps, 1}, \dots, \cB_d^{\veps, d})$ coincides with that of $(\cB_1, \dots, \cB_d)$, where $\cB_i^{\veps, i} \eqdef \{\bm{b}^{\veps, i} \mid \bm{b} \in \cB_i\}$. 

    After the enlargement, every intersection point $\bm{b}^{\veps, 1} \cap \dots \cap \bm{b}^{\veps, d} \, (\bm{b}^{\veps, i} \in \cB_i^{\veps, i})$ lies in the interior of the box $\bm{b}^{\veps, i}$ for each $i \in [d]$. So, we may shift each box $\bm{b} \in \cB_i^{\veps, i}$ along the $i$-th direction by a segment of length $\delta_{\bm{b}} \in \mathbb{R}$, which is far smaller than $\veps$, to separate all the boxes. Specifically, 
    \vspace{-0.5em}
    \begin{itemize}
        \item if $\bm{b}_1 \cap \dots \cap \bm{b}_d = \{p\}$, then $p + \delta_{\bm{b}_1} \vec{e}_1 + \dots + \delta_{\bm{b}_d} \vec{e}_d$ is the intersecting point of these boxes after shifting, where $\vec{e}_i$ denotes the $i$-th unit coordinate vector of $\R^d$ for $i \in [d]$; 
        \vspace{-0.5em}
        \item if $\bm{b}_1 \cap \dots \cap \bm{b}_d = \varnothing$, then after the shifting we still have $(\bm{b}_1 + \delta_{\bm{b}_1} \vec{e}_1) \cap \dots \cap (\bm{b}_d + \delta_{\bm{b}_d} \vec{e}_d) = \varnothing$, because all the boxes that we are considering in $\R^d$ are closed. 
    \end{itemize}
    \vspace{-0.5em}
    Therefore, we have separated the family $(\cB_1^{\veps,1}, \dots, \cB_d^{\veps,d})$ while preserving the intersection hypergraph. 

    We thus obtain $z_d(\vn; t) = z_d^*(\vn; t) \le M_d \cdot g_t(\vn)$, where ``$\le$'' follows from \Cref{prop:upper_separated}. 
\end{proof}

\subsection{From the canonical case to the general case} \label{sec:upper_general}

\begin{proof}[Proof of \Cref{thm:upper_general}]
    We proceed by induction on $d$. Let $\Lambda_r \eqdef \max\{4^{r-1}, M_r\}$, where $M_r > 0$ is the constant given by \Cref{prop:upper_separated}. Our goal is to prove \Cref{thm:upper_general} with this choice of $\Lambda_r$. 
    
    The base $d = 1$ case is shown by the following statement.

    \smallskip
    
    \begin{lemma} \label{lem:upper_onedim}
        Let $r \ge 2, \, t \ge 2$, and $\vF$ be an $r$-direction-vector in $\R^1$. We have $z_1^{\vF}(\vn; t) < 4^{r-1} \cdot g_t(\vn)$ for every $r$-size-vector $\vn$. 
    \end{lemma}
    
    The balanced case of the lemma, where all entries of $\vn$ are equal, follows from \cite[Proposition 8]{chan_keller_smorodinsky}. Since we require the unbalanced version for our reductive steps and the generalization is relatively straightforward, we defer the proof of \Cref{lem:upper_onedim} to \Cref{append:upper_onedim}. 

    \smallskip
    
    For the inductive step, let $\vn = (n_1, \dots, n_r)$ be an $r$-size-vector and $(\cB_1, \dots, \cB_r)$ be an $\vF$-family with $|\cB_j| = n_j \, (\forall j \in [r])$, where $\vF = (F_1, \dots, F_r)$ is a non-2-coherent $r$-direction-vector in $\mathbb{R}^d$. Given that the incidence graph $\cH$ of $(\cB_1, \dots, \cB_r)$ is $K_{t, \dots, t}^{(r)}$-free, we bound $e(\cH)$ in two cases. The first case will rely on the inductive hypothesis, while the second will follow from \Cref{prop:upper_canonical}. 

    \medskip
    
    \noindent\textbf{Case 1: there exist $i \in [d]$ and distinct $j_1, j_2 \in [r]$ such that $i \notin F_{j_1} \cup F_{j_2}$.} In this case, we may assume without loss of generality that $i = d$ and $(j_1, j_2) = (1, 2)$. Let $S$ be the collection of all reals $s$ such that the hyperplane $\Gamma_s\eqdef \{x_d = s\}$ contains at least one box $\bm{b} \in \cB_1 \cup \cB_2$. We shall apply the inductive hypothesis on each such hyperplane. 
    
    Write $F'_i \eqdef F_i \setminus \{d\}$ for $i = 1, \dots, r$ and $\vF' \eqdef (F'_1, \dots, F'_r)$. For each $s \in S$ and any $i \in [r]$, define
    \[
    \cB^{(s)}_i \eqdef \bigl\{ \bm{b} \cap \Gamma_s : \bm{b} \in \cB_i, \, \bm{b} \cap \Gamma_s \ne \varnothing \bigr\}. 
    \]
    Then $\bigl( \cB^{(s)}_1, \dots, \cB^{(s)}_r \bigr)$ is an $\vF'$-family in $\Gamma_s \cong \R^{d-1}$. Write $n_1^{(s)} \eqdef \bigl| \cB_1^{(s)} \bigr|, \, n_2^{(s)} \eqdef \bigl| \cB_2^{(s)} \bigr|$. It follows that
    \[
    \text{$\sum_{s \in S} n_1^{(s)} = n_1, \ \sum_{s \in S} n_2^{(s)} = n_2$, \, and $\bigl| \cB^{(s)}_i \bigr| \le n_i$ holds for $i = 3, 4, \dots, r$ and every $s \in S$}. 
    \]
    Let $\cH^{(s)}$ be the incidence hypergraph of $\bigl( \cB^{(s)}_1, \dots, \cB^{(s)}_r \bigr)$. Then $\cH^{(s)}$ is a subgraph of $\cH$, and hence it is again $K_{t, \dots, t}^{(r)}$-free. Since these subgraphs $\cH^{(s)} \, (s \in S)$ cover $\cH$, we obtain $e(\cH) \le \sum_{s \in S} e \bigl( \cH^{(s)} \bigr)$. 
    
    We first verify that the induction hypothesis is applicable to $\cH^{(s)}$. 
    
    \begin{claim}
        The direction vector $\vF'$ is non-2-coherent.
    \end{claim}

    {\renewcommand{\qedsymbol}{$\blacksquare$}
    \begin{proof}
        For any $k \in [r]$, from $F'_j \subseteq F_j \, (\forall j \in [r])$ we deduce that
        \[
        \Biggl| \bigcap_{j \in [r] \setminus \{k\}} F'_j \Biggr| \le \Biggl| \bigcap_{j \in [r] \setminus \{k\}} F_j \Biggr| \le 1. \qedhere
        \]
    \end{proof}
    }

    So, the inductive hypothesis implies that $e \bigl( \cH^{(s)} \bigr) \le \Lambda_r \cdot g_t \bigl( n_1^{(s)}, n_2^{(s)}, n_3, \dots, n_r \bigr)$. Therefore, 
    \begin{align*}
        e(\cH) &\le \sum_{s \in S} e \bigl( \cH^{(s)} \bigr) \le \Lambda_r \cdot \sum_{s \in S} g_t \bigl( n_1^{(s)}, n_2^{(s)}, n_3, \dots, n_r \bigr) \\
        &= \Lambda_r \cdot t \cdot \sum_{s \in S} n_1^{(s)} n_2^{(s)} n_3 \cdots n_r \cdot \biggl( \frac{1}{n_1^{(s)}} + \frac{1}{n_2^{(s)}} + \frac{1}{n_3} + \dots + \frac{1}{n_r} \biggr) \\
        &= \Lambda_r \cdot t \cdot \biggl( \sum_{s \in S} n_2^{(s)} n_3 \cdots n_r + \sum_{s \in S} n_1^{(s)} n_3 \cdots n_r + \sum_{s \in S} n_1^{(s)} n_2^{(s)} n_3 \cdots n_r \Bigl( \frac{1}{n_3}+\dots+\frac{1}{n_r} \Bigr) \biggr) \\
        &\le \Lambda_r \cdot t \cdot \biggl( n_2 n_3 \cdots n_r + n_1 n_3 \cdots n_r + n_1 n_2 n_3 \cdots n_r \Bigl( \frac{1}{n_3} + \dots + \frac{1}{n_r} \Bigr) \biggr) = \Lambda_r \cdot g_t(\vn), 
    \end{align*}
    where at the last ``$\le$'' we used the facts $\sum_{s \in S} n_1^{(s)} = n_1, \, \sum_{s \in S} n_2^{(s)} = n_2$, and $\sum_{s \in S} n_1^{(s)} n_2^{(s)} \le n_1 n_2$. 

    \medskip
    \noindent\textbf{Case 2: for every $i \in [d]$ there is at most one $j \in [r]$ such that $i\notin F_{j}$.} Denote $F_j^{\sc} \eqdef [d] \setminus F_j$ for $j = 1, \dots, r$. In this case, we know that $F_1^{\sc}, \dots, F_r^{\sc}$ are mutually disjoint. Moreover, we have the following equivalent formulation of the non-2-coherent condition on $\vF$:
    \begin{equation} \label{eq:coherent_def}
        \Biggl| \bigcap_{j \in [r] \setminus \{k\}} F_j \Biggr| \le 1 \, (\forall k \in [r]) \iff \Biggl| \bigcup_{j \in [r] \setminus \{k\}} F_j^{\sc} \Biggr| \ge d - 1 \, (\forall k \in [r]). 
    \end{equation}
    From~\eqref{eq:coherent_def} and the mutual disjointness of $F_j^{\sc}$, we see that $|F_j^{\sc}| \le 1$ for all $j \in [r]$. Moreover, $F_1^{\sc}, \dots, F_r^{\sc}$ have to cover $[d]$. Indeed, if not, then
    \vspace{-0.5em}
    \begin{itemize}
        \item either $F_j^{\sc} = \varnothing$ for all $j \in [r]$, which is impossible due to \eqref{eq:coherent_def} and $d \ge 2$, 
        \vspace{-0.5em}
        \item or there exists $k$ with $\abs{F_k^{\sc}} = 1$, yielding a contradiction against \eqref{eq:coherent_def}. 
    \end{itemize}
    \vspace{-0.5em}
    Therefore, without loss of generality, we may assume that \[
    F_j = \begin{cases}
        [d] \setminus \{j\} \qquad &\text{if $j \in \{1, \dots, d\} = [d]$}, \\
        [d] \qquad &\text{if $j \in \{d + 1, \dots, r\} = [r] \setminus [d]$}. 
    \end{cases}
    \]
    
    Now, we ``transfer'' the family $(\cB_1, \dots, \cB_r)$ into $\R^r$ in two steps as follows. 
    \vspace{-0.5em}
    \begin{itemize}
        \item First, embed each $\bm{b} \in \bigcup_{j=1}^r \cB_j$ into $\R^r$ by identifying it with $\bm{b} \times \underbrace{\{0\} \times \dots \times \{0\}}_{r-d}$. 
        \vspace{-0.5em}
        \item Next, for any $\bm{b} \in \cB_j$ with $j > d$, thicken it in the $(d+1)$-th through the $r$-th coordinates, except the $j$-th coordinate, by arbitrary lengths. 
    \end{itemize}
    \vspace{-0.5em}
    This process above yields a canonical family in $\R^r$ with exactly the same intersection hypergraph. Applying \Cref{prop:upper_canonical} to the ``transferred'' family, we obtain $e(\cH) \le M_r \cdot g_t(\vn)$. 

    \medskip
    Combining the above case analysis, the inductive proof of \Cref{thm:upper_general} is complete. 
\end{proof}

\section{Concluding remarks} \label{sec:remark}

Given an $\vF$-family of boxes $(\cB_1, \dots, \cB_r)$ in $\R^d$, we establish in \Cref{thm:main} a sharp dichotomy for $\vF$ that precisely characterizes when a polylogarithmic factor appears in $z_d^{\vF}(n; t)$. In particular, we have
\[
z_d^{\vF}(n; t) = \Omega \bigl( t n^{r-1} \tfrac{\log n}{\log\log n} \bigr)
\]
for every $2$-coherent $\vF$. However, it has been shown~\cite{chan_har-peled,chan_keller_smorodinsky} that $z_d^{\vF}(n; t) = \Theta \bigl( tn (\tfrac{\log n}{ \log\log n})^{d-1} \bigr)$ holds for $\vF = (\varnothing, [d])$ and that $z_d^{\vF}(n; t) = \Theta_r \bigl( tn^{r-1} (\tfrac{\log n}{ \log\log n})^{d-1} \bigr)$ holds for $\vF = \bigl( \underbrace{[d], \dots, [d]}_r \bigr)$. These naturally lead to the following question concerning the correct asymptotics of $z_d^{\vF}(n; t)$ in general. 

\begin{question} \label{exact2co}
    Determine the exact order of magnitude of $z_d^{\vF}(n; t)$ when $\vF$ is $2$-coherent.
\end{question}

To address \Cref{exact2co}, we single out the following two cases as fundamental instances. 

\begin{question}
    Determine the asymptotics of $z_3^{\vF}(n; t)$ for $\vF = (\{1\}, \{2, 3\})$. 
\end{question}

\begin{question}
    Determine the asymptotics of $z_4^{\vF}(n; t)$ for $\vF = (\{1, 2\}, \{3, 4\})$. 
\end{question}

\bibliographystyle{plain}
\bibliography{box_zarankiewicz}

@misc{chan_keller_smorodinsky_arxiv,
  title        = {On {Z}arankiewicz's problem for intersection hypergraphs of geometric objects},
  author       = {T. M. Chan and C. Keller and S. Smorodinsky},
  year         = {2024},
  eprint       = {2412.06490},
  archivePrefix= {arXiv},
  primaryClass = {math.CO},
  url          = {https://arxiv.org/abs/2412.06490},
  doi          = {10.48550/arXiv.2412.06490},
  note         = {\arXiv{2412.06490}}
}

@incollection {chan_keller_smorodinsky,
    AUTHOR = {Chan, T. M. and Keller, C. and Smorodinsky, S.},
     TITLE = {On {Z}arankiewicz's problem for intersection hypergraphs of geometric objects},
 BOOKTITLE = {41st {I}nternational {S}ymposium on {C}omputational {G}eometry},
    SERIES = {LIPIcs. Leibniz Int. Proc. Inform.},
    VOLUME = {332},
     PAGES = {Art. No. 33, 14},
 PUBLISHER = {Schloss Dagstuhl. Leibniz-Zent. Inform., Wadern},
      YEAR = {2025},
      ISBN = {978-3-95977-370-6},
   MRCLASS = {68U05},
  MRNUMBER = {4934386},
       DOI = {10.4230/lipics.socg.2025.33},
       URL = {https://doi.org/10.4230/lipics.socg.2025.33},
}

@incollection {keller_smorodinsky,
    AUTHOR = {Keller, C. and Smorodinsky, S.},
     TITLE = {Zarankiewicz's problem via {$\epsilon$}-$t$-nets},
 BOOKTITLE = {40th {I}nternational {S}ymposium on {C}omputational {G}eometry},
    SERIES = {LIPIcs. Leibniz Int. Proc. Inform.},
    VOLUME = {293},
     PAGES = {Art. No. 66, 15},
 PUBLISHER = {Schloss Dagstuhl. Leibniz-Zent. Inform., Wadern},
      YEAR = {2024},
      ISBN = {978-3-95977-316-4},
   MRCLASS = {68U05},
  MRNUMBER = {4757961},
       DOI = {10.4230/lipics.socg.2024.66},
       URL = {https://doi.org/10.4230/lipics.socg.2024.66},
}

@article {chan_har-peled,
    AUTHOR = {Chan, T. M. and Har-Peled, S.},
     TITLE = {On the number of incidences when avoiding an induced biclique in geometric settings},
   JOURNAL = {Discrete Comput. Geom.},
  FJOURNAL = {Discrete \& Computational Geometry. An International Journal of Mathematics and Computer Science},
    VOLUME = {73},
      YEAR = {2025},
    NUMBER = {2},
     PAGES = {466--489},
      ISSN = {0179-5376,1432-0444},
   MRCLASS = {05D10 (52C10)},
  MRNUMBER = {4865929},
MRREVIEWER = {Jonathan Tidor},
       DOI = {10.1007/s00454-024-00648-8},
       URL = {https://doi.org/10.1007/s00454-024-00648-8},
}

@article{basit_etal,
  author = {Basit, A. and Chernikov, A. and Starchenko, S. and Tao, T. and Tran, C.-M.},
  title = {Zarankiewicz’s problem for semilinear hypergraphs},
  journal = { Forum Math. Sigma},
  volume = {9},
  pages = {No. e59},
  year = {2021}
  }

@article{chazelle,
  author = {Chazelle, B.},
  title = {Lower bounds for orthogonal range searching: I. The reporting case},
  journal = {J. ACM},
  volume = {37},
  number = {2},
  pages = {200-212},
  year = {1990}
  }

@incollection {chalermsook_orgo_zarsav,
    AUTHOR = {Chalermsook, P. and Orgo, L. and Zarsav, M.},
     TITLE = {On geometric bipartite graphs with asymptotically smallest {Z}arankiewicz numbers},
 BOOKTITLE = {33rd {I}nternational {S}ymposium on {G}raph {D}rawing and {N}etwork {V}isualization},
    SERIES = {LIPIcs. Leibniz Int. Proc. Inform.},
    VOLUME = {357},
     PAGES = {Art. No. 21, 24},
 PUBLISHER = {Schloss Dagstuhl. Leibniz-Zent. Inform., Wadern},
      YEAR = {2025},
      ISBN = {978-3-95977-403-1},
   MRCLASS = {68R10},
  MRNUMBER = {4993088},
       DOI = {10.4230/lipics.gd.2025.21},
       URL = {https://doi.org/10.4230/lipics.gd.2025.21},
}

@article {kovari_sos_turan,
    AUTHOR = {K\"ovari, T. and S\'os, V. T. and Tur\'an, P.},
     TITLE = {On a problem of {K}. {Z}arankiewicz},
   JOURNAL = {Colloq. Math.},
  FJOURNAL = {Colloquium Mathematicum},
    VOLUME = {3},
      YEAR = {1954},
     PAGES = {50--57},
      ISSN = {0010-1354,1730-6302},
   MRCLASS = {27.2X},
  MRNUMBER = {65617},
MRREVIEWER = {J.\ Riguet},
       DOI = {10.4064/cm-3-1-50-57},
       URL = {https://doi.org/10.4064/cm-3-1-50-57},
}

@article {bukh2015,
    AUTHOR = {Bukh, B.},
     TITLE = {Random algebraic construction of extremal graphs},
   JOURNAL = {Bull. Lond. Math. Soc.},
  FJOURNAL = {Bulletin of the London Mathematical Society},
    VOLUME = {47},
      YEAR = {2015},
    NUMBER = {6},
     PAGES = {939--945},
      ISSN = {0024-6093,1469-2120},
   MRCLASS = {05C35 (05D99)},
  MRNUMBER = {3431574},
MRREVIEWER = {Lyuben\ R.\ Mutafchiev},
       DOI = {10.1112/blms/bdv062},
       URL = {https://doi.org/10.1112/blms/bdv062},
}

@article{zarankiewicz,
  author  = {Zarankiewicz, K.},
  title   = {Problem 101},
  journal = {Colloquium Mathematicum},
  volume  = {2},
  pages   = {301},
  year    = {1951}
}

@article {erdos1964,
    AUTHOR = {Erd{\H{o}}s, P.},
     TITLE = {On extremal problems of graphs and generalized graphs},
   JOURNAL = {Israel J. Math.},
  FJOURNAL = {Israel Journal of Mathematics},
    VOLUME = {2},
      YEAR = {1964},
     PAGES = {183--190},
      ISSN = {0021-2172},
   MRCLASS = {05.40},
  MRNUMBER = {183654},
MRREVIEWER = {A.\ H.\ Stone},
       DOI = {10.1007/BF02759942},
       URL = {https://doi.org/10.1007/BF02759942},
}

@article {fox_pach_sheffer_suk_zahl,
    AUTHOR = {Fox, J. and Pach, J. and Sheffer, A. and Suk, A. and Zahl, J.},
     TITLE = {A semi-algebraic version of {Z}arankiewicz's problem},
   JOURNAL = {J. Eur. Math. Soc. (JEMS)},
  FJOURNAL = {Journal of the European Mathematical Society (JEMS)},
    VOLUME = {19},
      YEAR = {2017},
    NUMBER = {6},
     PAGES = {1785--1810},
      ISSN = {1435-9855,1435-9863},
   MRCLASS = {05C35 (14P10 52C10 68R10)},
  MRNUMBER = {3646875},
MRREVIEWER = {J.\ C.\ Lagarias},
       DOI = {10.4171/JEMS/705},
       URL = {https://doi.org/10.4171/JEMS/705},
}

@misc{hunter_milojevic_sudakov_tomon,
  author = {Hunter, Z. and Milojevi\'{c}, A. and Sudakov, B. and Tomon, I.},
  title  = {${C}_4$-free subgraphs of high degree with geometric applications},
  year   = {2025},
  note   = {arXiv:2506.23942}
}

@misc{tidor_yu,
  author = {Tidor, J. and Yu, H.-H. H.},
  title  = {Multilevel polynomial partitioning and semialgebraic hypergraphs: regularity, {T}ur\'{a}n, and {Z}arankiewicz results},
  year   = {2024},
  note   = {arXiv:2407.20221}
}

@article {tomon,
    AUTHOR = {Tomon, I.},
     TITLE = {Coloring lines and {D}elaunay graphs with respect to boxes},
   JOURNAL = {Random Structures Algorithms},
  FJOURNAL = {Random Structures \& Algorithms},
    VOLUME = {64},
      YEAR = {2024},
    NUMBER = {3},
     PAGES = {645--662},
      ISSN = {1042-9832,1098-2418},
   MRCLASS = {05C62 (05C15 05D40 52C10)},
  MRNUMBER = {4724810},
MRREVIEWER = {Robert\ Cimikowski},
       DOI = {10.1002/rsa.21193},
       URL = {https://doi.org/10.1002/rsa.21193},
}

@article {szemeredi_trotter,
    AUTHOR = {Szemer\'edi, E. and Trotter, W. T.},
     TITLE = {Extremal problems in discrete geometry},
   JOURNAL = {Combinatorica},
  FJOURNAL = {Combinatorica. An International Journal of the J\'anos Bolyai Mathematical Society},
    VOLUME = {3},
      YEAR = {1983},
    NUMBER = {3-4},
     PAGES = {381--392},
      ISSN = {0209-9683},
   MRCLASS = {52A37 (05B25)},
  MRNUMBER = {729791},
MRREVIEWER = {Thomas\ O.\ Strommer},
       DOI = {10.1007/BF02579194},
       URL = {https://doi.org/10.1007/BF02579194},
}

@article {fox_pach,
    AUTHOR = {Fox, J. and Pach, J.},
     TITLE = {A separator theorem for string graphs and its applications},
   JOURNAL = {Combin. Probab. Comput.},
  FJOURNAL = {Combinatorics, Probability and Computing},
    VOLUME = {19},
      YEAR = {2010},
    NUMBER = {3},
     PAGES = {371--390},
      ISSN = {0963-5483,1469-2163},
   MRCLASS = {05C62 (05C10)},
  MRNUMBER = {2607373},
MRREVIEWER = {G\'eza\ T\'oth},
       DOI = {10.1017/S0963548309990459},
       URL = {https://doi.org/10.1017/S0963548309990459},
}

@article {mustafa_pach,
    AUTHOR = {Mustafa, N. H. and Pach, J.},
     TITLE = {On the {Z}arankiewicz problem for intersection hypergraphs},
   JOURNAL = {J. Combin. Theory Ser. A},
  FJOURNAL = {Journal of Combinatorial Theory. Series A},
    VOLUME = {141},
      YEAR = {2016},
     PAGES = {1--7},
      ISSN = {0097-3165,1096-0899},
   MRCLASS = {05C65 (05C35 05C62 05D40)},
  MRNUMBER = {3479234},
MRREVIEWER = {J\'ozsef\ Balogh},
       DOI = {10.1016/j.jcta.2016.02.001},
       URL = {https://doi.org/10.1016/j.jcta.2016.02.001},
}

@incollection {agarwal_erickson,
    AUTHOR = {Agarwal, P. K. and Erickson, J.},
     TITLE = {Geometric range searching and its relatives},
 BOOKTITLE = {Advances in discrete and computational geometry ({S}outh {H}adley, {MA}, 1996)},
    SERIES = {Contemp. Math.},
    VOLUME = {223},
     PAGES = {1--56},
 PUBLISHER = {Amer. Math. Soc., Providence, RI},
      YEAR = {1999},
      ISBN = {0-8218-0674-2},
   MRCLASS = {68U05},
  MRNUMBER = {1661376},
MRREVIEWER = {Hans-Dietrich\ Hecker},
       DOI = {10.1090/conm/223/03131},
       URL = {https://doi.org/10.1090/conm/223/03131},
}

@inproceedings {chan_zheng,
    AUTHOR = {Chan, T. M. and Zheng, D. W.},
     TITLE = {Hopcroft's problem, log-star shaving, 2{D} fractional cascading, and decision trees},
 BOOKTITLE = {Proceedings of the 2022 {A}nnual {ACM}-{SIAM} {S}ymposium on {D}iscrete {A}lgorithms ({SODA})},
     PAGES = {190--210},
 PUBLISHER = {[Society for Industrial and Applied Mathematics (SIAM)], Philadelphia, PA},
      YEAR = {2022},
      ISBN = {978-1-61197-707-3},
   MRCLASS = {68U05},
  MRNUMBER = {4415047},
       DOI = {10.1137/1.9781611977073.10},
       URL = {https://doi.org/10.1137/1.9781611977073.10},
}

@article {bukh2024,
    AUTHOR = {Bukh, B.},
     TITLE = {Extremal graphs without exponentially small bicliques},
   JOURNAL = {Duke Math. J.},
  FJOURNAL = {Duke Mathematical Journal},
    VOLUME = {173},
      YEAR = {2024},
    NUMBER = {11},
     PAGES = {2039--2062},
      ISSN = {0012-7094,1547-7398},
   MRCLASS = {05D40 (05C55 14N07)},
  MRNUMBER = {4779871},
MRREVIEWER = {Nikolaos\ Fountoulakis},
       DOI = {10.1215/00127094-2023-0043},
       URL = {https://doi.org/10.1215/00127094-2023-0043},
}

@article {kollar_ronyai_szabo,
    AUTHOR = {Koll\'ar, J. and R\'onyai, L. and Szab\'o, T.},
     TITLE = {Norm-graphs and bipartite {T}ur\'an numbers},
   JOURNAL = {Combinatorica},
  FJOURNAL = {Combinatorica. An International Journal on Combinatorics and
              the Theory of Computing},
    VOLUME = {16},
      YEAR = {1996},
    NUMBER = {3},
     PAGES = {399--406},
      ISSN = {0209-9683},
   MRCLASS = {05C35},
  MRNUMBER = {1417348},
MRREVIEWER = {W.\ G.\ Brown},
       DOI = {10.1007/BF01261323},
       URL = {https://doi.org/10.1007/BF01261323},
}

@article {alon_ronyai_szabo,
    AUTHOR = {Alon, N. and R\'onyai, L. and Szab\'o, T.},
     TITLE = {Norm-graphs: variations and applications},
   JOURNAL = {J. Combin. Theory Ser. B},
  FJOURNAL = {Journal of Combinatorial Theory. Series B},
    VOLUME = {76},
      YEAR = {1999},
    NUMBER = {2},
     PAGES = {280--290},
      ISSN = {0095-8956,1096-0902},
   MRCLASS = {05C35 (05C15 05D05)},
  MRNUMBER = {1699238},
MRREVIEWER = {W.\ G.\ Brown},
       DOI = {10.1006/jctb.1999.1906},
       URL = {https://doi.org/10.1006/jctb.1999.1906},
}

@misc{chen_liu_ye,
      title={Extremal constructions for apex partite hypergraphs}, 
      author={Chen, Q. and Liu, H. and Ye, K.},
      year={2025},
      eprint={2510.07997},
      archivePrefix={arXiv},
      primaryClass={math.CO},
      url={https://arxiv.org/abs/2510.07997}, 
      note={ar{X}iv:2510.07997}
}

@article {ma_yuan_zhang,
    AUTHOR = {Ma, J. and Yuan, X. and Zhang, M.},
     TITLE = {Some extremal results on complete degenerate hypergraphs},
   JOURNAL = {J. Combin. Theory Ser. A},
  FJOURNAL = {Journal of Combinatorial Theory. Series A},
    VOLUME = {154},
      YEAR = {2018},
     PAGES = {598--609},
      ISSN = {0097-3165,1096-0899},
   MRCLASS = {05C35 (05C65)},
  MRNUMBER = {3718079},
MRREVIEWER = {Bal\'azs\ Patk\'os},
       DOI = {10.1016/j.jcta.2017.10.002},
       URL = {https://doi.org/10.1016/j.jcta.2017.10.002},
}

\appendix

\section{Proof of \texorpdfstring{\Cref{lem:upper_onedim}}{Lemma 4.7}} \label{append:upper_onedim}

In $\R^1$, there are only two possible direction sets, $\varnothing$ or $\{1\}$. A box family with direction set $\varnothing$ is a collection of points, and a box family with direction set $\{1\}$ is a collection of segments. To prove \Cref{lem:upper_onedim}, we begin with the $r = 2$ case. Firstly, we establish the result for a set of points and a set of segments. Secondly, we establish the result for two sets of segments. From here, we can deduce \Cref{lem:upper_onedim} for any $r$ sets of segments. Finally, we can thicken all the points in $\cB_i$ with direction sets $\varnothing$ and reduce the general case to the case of $r$ sets of segments.

Throughout this appendix, every set refers to a multiset. 

\smallskip

We begin with the case of points and segments, corresponding to $d = 1$ in \cite[Lemma~2.1]{chan_har-peled}. We include a proof here to remove the assumption $m = \operatorname{poly}(n)$ in their statement. 

\begin{proposition} \label{prop:point_segment}
    In $\R^1$, let $\cA$ be a set of $m$ points and $\cB$ be a set of $n$ segments. Denote by $G$ the intersection graph between $\cA$ and $\cB$. If $G$ is $K_{s, t}$-free ($s$ from the $\cA$-side), then $e(G) < mt + 3ns$. 
\end{proposition}

\begin{proof}
    

    Order $m$-point set $\cA \subset \R^1$ as $a_1 \le \dots \le a_m$. Set $\ell \eqdef \lceil \frac{m}{s} \rceil$ and partition $\cA$ into $\cA_1, \dots, \cA_\ell$:
    \[
    \cA_i \eqdef \begin{cases}
        \bigl\{ a_{(i-1)s+1}, a_{(i-1)s+2}, \dots, a_{is} \bigr\} \qquad &\text{if $i \in [\ell-1]$}, \\
        \bigl\{ a_{(\ell-1)s+1}, a_{(\ell-1)s+2}, \dots, a_m \bigr\} \qquad &\text{if $i = \ell$}. 
    \end{cases}
    \]
    Write $\cB_i \eqdef \{\bm{b} \in \cB : \bm{b} \cap \cA_i \ne \varnothing\}$ and denote by $G_i$ the intersection graph between $\cA_i$ and $\cB_i$. From the definitions we deduce that $e(G) \le e(G_1) + \dots + e(G_{\ell})$. 

    It remains to bound $e(G_i)$ for $i = 1, \dots, \ell$. If $i = \ell$, then $e(G_{\ell}) \le ns$. Suppose $i \in [\ell-1]$ then. Since $G$ is $K_{s, t}$-free, the number of segments in $\cB_i$ covering the interval $I_i \eqdef [\min\cA_i, \max\cA_i]$ is at most $t - 1$. Let $n_i$ be the number of segments in $\cB_i$ not covering $I_i$. A crucial observation is that $n_1 + \dots + n_{\ell-1} \le 2n$. Indeed, every segment in $\cB$ has two endpoints, and each endpoint appears in at most one interior of $I_1, \dots, I_{\ell-1}$ because the interiors are disjoint. Since $\ell - 1 \le \frac{m}{s}$, we obtain
    \begin{align*}
        e(G) \le \sum_{i=1}^{\ell} e(G_i) = e(G_{\ell}) + \sum_{i=1}^{\ell-1} e(G_i) &\le ns + \sum_{i=1}^{\ell-1} \bigl( (t-1)s + n_i(s-1) \bigr) \\
        &\le ns + (\ell-1)(t-1)s + 2n(s-1) < mt + 3ns. \qedhere
    \end{align*}
\end{proof}

We then consider two sets of segments. Our proof is inspired by the approach of \cite[Theorem~7]{keller_smorodinsky}, reducing geometric intersections to lower-dimensional boundary intersections. 

\begin{proposition} \label{prop:segment_segment}
    In $\R^1$, let $\cA$ and $\cB$ be sets of $m$ and $n$ segments, respectively. Denote by $G$ the intersection graph between $\cA$ and $\cB$. If $G$ is $K_{s, t}$-free ($s$ from the $\cA$-side), then $e(G) < 4(mt + ns)$. 
\end{proposition}

\begin{proof}
    Let $\partial \cA$ be the multiset of the $2m$ endpoints of $m$ segments in $\cA$, and denote by $G_A$ the intersection graph between $\partial \cA$ and $\cB$. If there were $2s - 1$ points in $\partial \cA$ all of which are contained in a same $t$ segments in $\cB$, then the pigeonhole principle would show that these $2s - 1$ points belong to at least $s$ distinct segments in $\cA$, yielding a copy of $K_{s, t}$ in $G$. This is impossible because $G$ is $K_{s, t}$-free. It follows that $G_A$ is $K_{2s-1, t}$-free, and so \Cref{prop:point_segment} implies that 
    \[
    e(G_A) <  m  t + 3n (2s-1)  < mt + 6ns. 
    \]

    Let $\partial \cB$ be the set of the $2n$ endpoints of $n$ segments in $\cB$. Denote by $G_B$ the intersection graph between $\cA$ and $\partial \cB$. As before, $G_B$ is $K_{s, 2t-1}$-free, and so $e(G_B) < 6mt + ns$. Note that
    \vspace{-0.5em}
    \begin{itemize}
        \item each incidence in $G$ induces at least two incidences in $G_A$ and $G_B$, and
        \vspace{-0.5em}
        \item each incidence in $G_A$ or $G_B$ uniquely determines the incidence in $G$ that induces it. 
    \end{itemize}
    \vspace{-0.5em}
    Therefore, by double counting, we obtain $e(G) \le \frac{e(G_A) + e(G_B)}{2} < 4(mt + ns)$. 
\end{proof}

Next, we consider the case of $r$ sets of segments. Here, we will use the same strategy as in the proof of~\cite[Proposition 8]{chan_keller_smorodinsky}. Recall that $g_t(n_1, \dots, n_r) = t n_1 \cdots n_r \cdot \bigl( \frac{1}{n_1} + \dots + \frac{1}{n_r} \bigr) \in \N_+$. 

\begin{proposition} \label{prop:r_segment}
    Let $r \ge 2$ and $t \ge 2$. In $\R^1$, let $\cB_1, \dots, \cB_r$ be $r$ sets of segments with $|\cB_i| = n_i$ for $i = 1, \dots, r$. Denote by $\cH$ the intersection hypergraph between $\cB_1, \dots, \cB_r$. If $\cH$ is $K_{t, \dots, t}^{(r)}$-free, then the number of hyperedges $e(\cH) < 4^{r-1} \cdot g_t(\vn)$ holds for every $r$-size-vector $\vn$. 
    
    In other words, $z_1^{\vF}(\vn; t) < 4^{r-1} \cdot g_t(\vn)$ holds for any $r$-size-vector $\vn$, where $\vF = \bigl( \underbrace{\{1\}, \dots, \{1\}}_{r} \bigr)$. 
\end{proposition}
  
\begin{proof}
    We proceed by induction on $r \ge 2$. The base $r = 2$ case follows from \Cref{prop:segment_segment}. 
    
    Suppose henceforth that $r \ge 3$ and that the statement has been established for $r - 1$. Define
    \[
    \widetilde{\cB} = \cB_1 \cap \dots \cap \cB_{r-1} \eqdef \biggl\{ \bm{b}_1 \cap \dots \cap \bm{b}_{r-1} : \text{$\bm{b}_i \in \cB_i$ with $\bigcap\limits_{i=1}^{r-1} \bm{b}_i \ne \varnothing$} \biggr\}
    \]
    as a multiset of segments. Let $G$ be the intersection graph between $\widetilde{\cB}$ and $\cB_r$. Then $e(G) = e(\cH)$. 

    We claim that $G$ is $K_{T, t}$-free ($T$ from the $\widetilde{\cB}$-side), where $T \eqdef 4^{r-2} \cdot g_t(n_1, \dots, n_{r-1})$. Suppose to the contrary that $G$ contains a copy of $K_{T, t}$. Then there are $t$ segments $\bm{b}_{r, 1}, \dots, \bm{b}_{r, t} \in \cB_r$, each intersecting a fixed collection of $T$ distinct segments of the form $\bm{b}_{1, j} \cap \dots \cap \bm{b}_{r-1, j}$, where $\bm{b}_{i, j} \in \cB_i$ for $i \in [r-1]$ and $j \in [T]$. Let $\cB_r^* \eqdef \{\bm{b}_{r, 1}, \dots, \bm{b}_{r, t}\}$, and for each $i \in [r-1]$ define $\cB_i' \eqdef \bigl\{ \bm{b}_{i, j} : j \in [T] \bigr\}$, viewed as a submultiset of $\cB_i$ (counting distinct elements separately, even if they represent the same interval). Let $\cH'$ be the $(r-1)$-partite $(r-1)$-uniform intersection hypergraph of $\cB_1', \dots, \cB_{r-1}'$. We have $n_i' \eqdef |\cB_i'| \le n_i$ and $e(\cH') \ge T \ge 4^{r-1} \cdot g_t(n_1', \dots, n_{r-1}')$. Due to the inductive hypothesis, $\cH'$ contains $K_{t, \dots, t}^{(r-1)}$. So, for each $i \in [r-1]$, there exists a submultiset $\cB_i^* \subseteq \cB_i'$ of size $t$ such that every choice $\bm{b}_1^* \in \cB_1^*, \dots, \bm{b}_{r-1}^* \in \cB_{r-1}^*$ satisfies $\bm{b}_1^* \cap \dots \cap \bm{b}_{r-1}^* \ne \varnothing$. Since $\cB_i^* \subseteq \cB_i'$, each $\bm{b}_i^* \in \cB_i^*$ intersects every $\bm{b}_r^* \in \cB_r^*$. Hence, the fact that boxes have Helly number $2$ (\Cref{obs:box_helly}) implies that the families $\cB_1^*,\dots,\cB_r^*$ form a $K_{t,\dots,t}^{(r)}$ in $\cH$, a contradiction. Therefore, $G$ is $K_{T,t}$-free.
    
    Since there are at most $n_1 \cdots n_{r-1}$ segments in $\widetilde{\cB}$ and $n_r$ segments in $\cB_r$, \Cref{prop:segment_segment} shows
    \[
    e(\cH) = e(G) < 4(n_1 \cdots n_{r-1} \cdot t + n_r \cdot T) < 4^{r-1} \cdot g_t(\vn). 
    \]
    The inductive proof of \Cref{prop:r_segment} is complete. 
\end{proof}

\begin{proof}[Proof of \Cref{lem:upper_onedim}]
    Suppose $\vF = (F_1, \dots, F_r) \in \bigl\{ \varnothing, \{1\} \bigr\}^r$. For any $i \in [r]$ with $F_i = \varnothing$, we replace every point $p$ in $\cB_i$ by a segment containing $p$. Take the replacing segments to be sufficiently small so that all incidence relations remain intact. According to \Cref{prop:r_segment}, we conclude that
    \[
    z_1^{\vF}(\vn; t) \le z_1^{\vF'}(\vn; t) < 4^{r-1} g_t(\vn),
    \]
    where $\vF'= \bigl( \underbrace{\{1\}, \dots, \{1\}}_{r} \bigr)$.
\end{proof}

\end{document}